\documentclass[11pt,a4paper]{article}
\usepackage{amsthm}
\usepackage{amsmath}
\usepackage{authblk}
\usepackage[makeroom]{cancel}
\usepackage[tablesonly,nolists]{endfloat}
\usepackage[hmargin={28mm,28mm},vmargin={30mm,35mm}]{geometry}
\usepackage[colorlinks,allcolors={blue}]{hyperref}
\usepackage{newtxtext,newtxmath}
\usepackage{nicefrac}
\usepackage{pgfplots}
\usepackage{tikz}
\usepackage{booktabs}
\usetikzlibrary{spy}
\usepackage{paralist}
\usepackage{rotating}
\usepackage{stmaryrd}
\usepackage{subcaption}
\usepackage[compact,small]{titlesec}
\usepackage{multirow} 
\usepackage[fulladjust]{marginnote}
\usepackage[backend=bibtex,style=numeric-comp,maxbibnames=99,giveninits=true]{biblatex}
\bibliography{nsho_pr_poly}

\newcommand{\email}[1]{\href{mailto:#1}{#1}}

\newcommand{\be}{\boldsymbol{e}}
\newcommand{\bef}{\boldsymbol{f}}
\newcommand{\bg}{\boldsymbol{g}}
\newcommand{\bq}{\boldsymbol{q}}

\newcommand{\bn}{\boldsymbol{n}}
\newcommand{\bu}{\boldsymbol{u}}
\newcommand{\bv}{\boldsymbol{v}}
\newcommand{\bw}{\boldsymbol{w}}
\newcommand{\bx}{\boldsymbol{x}}
\newcommand{\by}{\boldsymbol{y}}
\newcommand{\bz}{\boldsymbol{z}}

\newcommand{\bE}{\boldsymbol{E}}
\newcommand{\bF}{\boldsymbol{F}}
\newcommand{\bG}{\boldsymbol{G}}
\newcommand{\bI}{\boldsymbol{I}}
\newcommand{\bR}{\boldsymbol{R}}
\newcommand{\bU}{\boldsymbol{U}}

\newcommand{\bvs}{\boldsymbol{\mathfrak{v}}}
\newcommand{\bws}{\boldsymbol{\mathfrak{w}}}

\newcommand{\qs}{{\mathfrak{\phi}}}
\newcommand{\ps}{{\mathfrak{\psi}}}
\newcommand{\ts}{\boldsymbol{\mathfrak{\zeta}}}
\newcommand{\xis}{\boldsymbol{\mathfrak{\xi}}}

\newcommand{\Poly}[1]{\mathcal{P}^{#1}}

\newcommand{\Goly}[1]{\boldsymbol{\mathcal{G}}^{#1}}

\newcommand{\RTN}[1]{\boldsymbol{\mathcal{RT}}^{#1}}

\newcommand{\Mh}[1][h]{\mathcal{M}_{#1}}
\newcommand{\Th}[1][h]{\mathcal{T}_{#1}}
\newcommand{\Fh}[1][h]{\mathcal{F}_{#1}}
\newcommand{\Fhi}[1][h]{\mathcal{F}_{#1}^{{\rm i}}}
\newcommand{\Fhb}[1][h]{\mathcal{F}_{#1}^{{\rm b}}}

\newcommand{\Mhs}[1][h]{\mathfrak{M}_{#1}}
\newcommand{\Ths}[1][h]{\mathfrak{T}_{#1}}
\newcommand{\Fhs}[1][h]{\mathfrak{F}_{#1}}
\newcommand{\Fhsi}[1][h]{\mathfrak{F}_{#1}^{{\rm i}}}

\newcommand{\frakR}{\boldsymbol{\mathfrak{R}}}

\newcommand{\normal}{\bn}

\newcommand{\uline}[1]{\underline{#1}}
\newcommand{\Real}{\mathbb R}

\newcommand{\GRAD}{\nabla}
\newcommand{\DIV}{\nabla\cdot}
\newcommand{\ROT}{\nabla\times}
\newcommand{\ROTh}{\widehat{\nabla}\times}
\newcommand{\LAP}{\Delta}

\newcommand{\Ldeux}[1][\Omega]{{L}^2({#1})}
\newcommand{\Ldeuxd}[1][\Omega]{{L}^2({#1})^3}
\newcommand{\Ldeuxz}[1][\Omega]{{L}^2_0({#1})}

\newcommand{\Hun}[1][\Omega]{{H}^1({#1})}
\newcommand{\Ldeuxx}[1][\Omega]{{L}^2({#1})}
\newcommand{\Ldeuxxz}[1][\Omega]{{L}^2_0({#1})}

\newcommand{\Hund}[1][\Omega]{{H}^1({#1})^3}

\newcommand{\Hdiv}[1][\Omega]{{\boldsymbol{H}}_{\text{div}}({#1})}

\newcommand{\norm}[2]{\|#2\|_{#1}}
\newcommand{\seminorm}[2]{|#2|_{#1}}

\newcommand{\nnorm}[2][]{\|#2\|_{#1}}

\newcommand{\Reynolds}{\mathrm{Re}}
\newcommand{\bH}{\boldsymbol{H}}
\newcommand{\calF}{{\mathcal F}}
\newcommand{\calT}{{\mathcal T}}

\newcommand{\mxJ}{\boldsymbol{\mathbb J}}
\newcommand{\mxA}{\boldsymbol{\mathbb A}}

\newcommand{\bpsi}{{\boldsymbol \psi}}
\newcommand{\bpi}{{\boldsymbol \pi}}
\newcommand{\btau}{{\boldsymbol \tau}}

\newcommand{\bphi}{{\boldsymbol \phi}}

\DeclareMathOperator{\card}{card}
\newcommand*\xbar[1]{%
  \hbox{%
    \vbox{%
      \hrule height 0.75pt % The actual bar
      \kern0.5ex%         % Distance between bar and symbol
      \hbox{%
        \kern-0.1em%      % Shortening on the left side
        \ensuremath{#1}%
        \kern-0.1em%      % Shortening on the right side
      }%
    }%
  }%
} 
\usepackage[normalem]{ulem}
\newcounter{corr}
\definecolor{violet}{rgb}{0.580,0.,0.827}
\newcommand{\corr}[3]{\typeout{Warning : a correction remains in page
    \thepage}
				\stepcounter{corr}        
				{\color{blue}\ifmmode\text{\,\sout{\ensuremath{#1}}\,}\else\sout{#1}\fi}
        {\color{red}#2}
        {\color{violet} #3}}
\newtheorem{theorem}{Theorem}
\newtheorem{proposition}[theorem]{Proposition}
\newtheorem{lemma}[theorem]{Lemma}

\theoremstyle{remark}
\newtheorem{remark}[theorem]{Remark}
\theoremstyle{definition}

\newcommand{\figpath}{./figures}

\begin{document}
  \title{A pressure-robust HHO method for the solution of the incompressible Navier--Stokes equations on general meshes}

\author[1,2,3]{Daniel Castanon Quiroz\footnote{\email{danielcq.mathematics@gmail.com}}}
\author[3]{Daniele A. Di Pietro\footnote{\email{daniele.di-pietro@umontpellier.fr}}}

\affil[1]{Instituto de Investigaciones en Matemáticas Aplicadas y en Sistemas, Universidad Nacional Autónoma de México, Circuito Escolar s/n, Ciudad Universitaria C.P. 04510 Cd. Mx. (México)
}

\affil[2]{Universit\'e C\^ote d'Azur, CNRS, Inria team Coffee, LJAD, Nice, France}
\affil[3]{IMAG, Univ Montpellier, CNRS, Montpellier, France}\maketitle

\begin{abstract}
  In a recent work \cite{Castanon-Quiroz.Di-Pietro:20}, we have introduced a pressure-robust Hybrid High-Order method for the numerical solution of the incompressible Navier--Stokes equations on matching simplicial meshes.
  Pressure-robust methods are characterized by error estimates for the velocity that are fully independent of the pressure.
  A crucial question was left open in that work, namely whether the proposed construction could be extended to general polytopal meshes.
  In this paper we provide a positive answer to this question.
  Specifically, we introduce a novel divergence-preserving velocity reconstruction that hinges on the solution inside each element of a mixed problem on a subtriangulation, then use it to design discretizations of the body force and convective terms that lead to pressure robustness. An in-depth theoretical study of the properties of this velocity reconstruction, and their reverberation on the scheme, is carried out for arbitrary polynomial degrees $k\geq 0$ and meshes composed of general polytopes.
The theoretical convergence estimates and the pressure robustness of the method are confirmed by an extensive panel of numerical examples.

\end{abstract}

%------------------------------------------------------------------------------%
%------------------------------------------------------------------------------%

\section{Introduction}

This paper focuses on numerical approximations of the Navier--Stokes equations robust with respect to large irrotational body forces.
Specifically, we address a nontrivial question left open in the previous work \cite{Castanon-Quiroz.Di-Pietro:20}, namely whether robustness can be achieved on general polyhedral meshes such as the ones supported by the Hybrid High-Order (HHO) method \cite{Di-Pietro.Ern:15,Di-Pietro.Droniou:20}.
\smallskip

Let $\Omega \subset \mathbb{R}^3$ denote an open, bounded, simply connected polyhedral domain with Lipschitz boundary $\partial \Omega$.
Let $\nu > 0$ be the kinematic viscosity of the fluid and $\bef\in L^2(\Omega)^3$ a given vector field representing a body force.
Setting $\bU\coloneq  H_0^1(\Omega)^3$ and $P\coloneq \Ldeuxz= \left\{q \in \Ldeux : \int_{\Omega}q=0\right\}$, we consider the Navier--Stokes problem:
Find $(\bu,p)\in\bU\times P$ such that
\begin{subequations}
  \label{eq:nstokes:weak}
  \begin{alignat}{2}
    \nu\int_\Omega\nabla \bu :  \nabla \bv
    + \int_\Omega ((\ROT \bu) \times \bu) \cdot \bv
    - \int_\Omega (\DIV \bv) p
    &= \int_\Omega \bef\cdot\bv
    &\qquad&  \forall \bv \in \bU,
    \label{eq:nstokes:weak:momentum}\\
    \int_\Omega (\DIV \bu) q
    &=0
    &\qquad& \forall q \in\Ldeux.
    \label{eq:nstokes:weak:mass}
  \end{alignat}
\end{subequations}
Above, $\DIV$ and $\ROT$ denote, respectively, the divergence and curl operators, while $\times$ is the cross product of two vectors.
The convective term in \eqref{eq:nstokes:weak:momentum} is expressed in rotational form, so $p$ is here the Bernoulli pressure, which is related to the kinematic pressure $p_{\text{kin}}$ by the equation $p=p_{\text{kin}} + \frac{1}{2}|\bu|^2$.
\smallskip

The domain $\Omega$ being simply connected, we have the following Hodge decomposition of the body force (see, e.g., \cite[Section 4.3]{Arnold:18}):
\begin{equation}\label{eq:hodge.f}
  \bef = \bg + \lambda\GRAD\psi,
\end{equation}
where $\bg$ is the curl of a function in $\bH({\bf curl};\Omega)$ the tangent trace of which vanishes on $\partial\Omega$,
$\psi\in H^1(\Omega)$ is such that $\|\GRAD\psi\|_{L^2(\Omega)^3}=1$, and $\lambda\in\Real^+$.
It is well-known that, at the continuous level, the velocity field is entirely determined by the first component in the decomposition \eqref{eq:hodge.f}.
This property, however, does not carry out automatically to the discrete level.
The development of numerical methods that possess this property, and which are sometimes referred to in the literature as \emph{pressure-robust}, has been an active field of research over the last few years;
see, e.g., \cite{Falk.Neilan:13,Linke:14,Linke.Merdon:16*1,John.Linke.ea:17,Ahmed.Linke.ea:18} concerning finite element methods on standard meshes.
\smallskip

Recently, the mathematical community have become interested in the development of arbitrary-order approximation methods that support more general meshes than standard finite elements and which can include, e.g., polyhedral elements and non-matching interfaces.
A representative but by far non exhaustive list of references concerning incompressible flow problems includes \cite{Di-Pietro.Ern:10,Di-Pietro.Ern:12,Di-Pietro.Krell:18,Gatica.Munar.ea:18,Botti.Di-Pietro.ea:19,Beirao-da-Veiga.Lovadina.ea:18,Beirao-da-Veiga.Dassi.ea:20,Zhang.Zhao.ea:21}; see also the recent works \cite{Botti.Castanon-Quiroz.ea:21,Castanon-Quiroz.Di-Pietro.ea:21} concerning non-Newtonian fluids.
Pressure-robust variations of the HHO method on matching simplicial meshes for the Stokes and Navier--Stokes problem have been proposed, respectively, in \cite{Di-Pietro.Ern.ea:16,Castanon-Quiroz.Di-Pietro:20}.

The development of pressure-robust methods on polyhedral meshes is, however, a challenging task.
Some of the first genuinely pressure-robust polyhedral methods for the Stokes equations have been proposed in \cite{Liu.Harper.ea:20,Wang.Mu.ea:21,Zhao.Park.ea:20}.
These methods handle the lowest order case using a velocity reconstruction in $\bH({\rm div};\Omega)$ introduced in \cite{Chen.Wang:17} and relying on Wachspress (generalized barycentric) coordinates.
This approach has two shortcomings:
first, the faces of each (convex) polyhedral element must be either triangles or parallelograms;
second, error estimates for the approximated velocity would require gradient bounds for the Wachspress coordinates on an arbitrary convex polyhedron, the derivation of which remains, to the best of our knowledge, an open problem.
Regarding arbitrary-order methods on general meshes, a pressure-robust Virtual Element method has been recently proposed in \cite{Frerichs.Merdon:20} for the Stokes equations.
The extension of this method to the Navier--Stokes equations remains, to the best of our knowledge, an open problem.
A pressure-robust discretization scheme for the full Navier--Stokes equations has been proposed in \cite{Kim.Zhao.ea:21} based on the staggered Discontinuous Garlekin method. 
This method solves for three unknowns (the pressure, the velocity, and its gradient), thus leading to larger algebraic systems.
Recently, a novel HHO method for which pressure-robustness has been numerically demonstrated has been proposed in \cite{Botti.Massa:21}.
This method uses a larger pressure space than the one considered in the present work, and the derivation of rigorous pressure-robust error estimates is still to be done.
An entirely different approach to pressure-robustness on polyhedral meshes has also been recently pursued in \cite{Beirao-da-Veiga.Dassi.ea:21}, hinging on the compatibility features of Discrete de Rham \cite{Di-Pietro.Droniou.ea:20,Di-Pietro.Droniou:21} and Virtual Element methods.
While this approach leads to a fully pressure-robust, arbitrary-order method, it is based on a curl-curl formulation of the viscous term, which does not lend itself naturally to the treatment of certain standard boundary conditions.
\medskip

In the present work, we propose a novel fully pressure-robust HHO method for the Navier--Stokes problem \eqref{eq:nstokes:weak} that works in space dimension two and three and supports general meshes composed of  polytopal elements.
The cornerstone of the method is a local divergence-preserving reconstruction of the velocity built inside each mesh element $T$ by solving a mixed problem inspired by {\cite{Kuznetsov.Repin:04,Kuznetsov.Repin:05,Lederer:2017}} on a subtriangulation of $T$; see also \cite{Vohralik.Wohlmuth:13}.
The assumptions made in Section   \ref{sec:setting:mesh} for each element  $T$ enable us to derive the required continuity and approximation bounds for this reconstruction.
Robustness with respect to large irrotational body forces is achieved by leveraging the divergence-preserving velocity reconstruction in the discretisation of both the convective term and the body force, so that similar properties as the ones discussed in \cite[Section 4.3 and Lemma 7]{Castanon-Quiroz.Di-Pietro:20} are obtained for these terms.
\smallskip

The rest of the paper is organised as follows.
In Section \ref{sec:setting} we introduce the discrete setting, including mesh assumptions, notation, and the novel divergence-preserving velocity reconstruction.
Section \ref{sec:discrete.problem} contains the discrete problem and the main results of the analysis, with particular focus on the definition and properties of the discrete convective trilinear form.
A complete panel of two-dimensional numerical tests on a variety of polygonal meshes is provided in Section \ref{sec:tests}, including a comparison with the standard HHO scheme of \cite{Botti.Di-Pietro.ea:19}.

%------------------------------------------------------------------------------%
%------------------------------------------------------------------------------%

\section{Discrete setting}\label{sec:setting}

The following exposition focuses on the three-dimensional case $d=3$, the two-dimensional case $d=2$ being a special instance of the latter as detailed in Remark \ref{rem:2d} below.

\subsection{Mesh}\label{sec:setting:mesh}

Following \cite[Definition 1.4]{Di-Pietro.Droniou:20}, we consider a polyhedral mesh defined as a couple $\Mh\coloneq(\Th,\Fh)$,
where $\Th$ is a finite collection of polyhedral elements which we additionally assume to be convex 
({ see Remark \ref{rem:convex.a} below on how to relax this assumption}), while $\Fh$ is a finite collection of planar faces $F$.
For any mesh element or face $X\in\Th\cup\Fh$, we denote by $|X|$ its Hausdorff measure and by $h_X$ its diameter, so that the meshsize satisfies $h = \max_{T\in\Th}h_T$.
Boundary faces lying on $\partial\Omega$ and internal faces contained in
$\Omega$ are collected in the sets $\Fhb$ and $\Fhi$, respectively.
For each mesh element $T\in\Th$, we denote by $\Fh[T]$ the set collecting the faces that lie on the boundary $\partial T$ of $T$ and, for all $F\in\Fh[T]$, we denote by $\normal_{TF}$ the (constant) unit vector normal to $F$ and pointing out of $T$.

It is assumed that $\Mh$ belongs to a regular mesh sequence $(\Mh)_h$ in the sense of \cite[Definition 1.9]{Di-Pietro.Droniou:20}.
This assumption entails the existence of a matching simplicial submesh $\Mhs\coloneq(\Ths,\Fhs)$ of $\Mh$ with the following properties:
$\Ths$ is a finite collection of simplicial elements;
for any simplex $\tau \in \Ths$ , there is a unique mesh element $T \in \Th$ such that $\tau \subset T$;
for any simplicial face $\sigma \in \Fhs$ and any mesh face $F \in \Fh$ , either $\sigma \cap F = \emptyset $ or $\sigma \subset F$.
For $T\in \Th$, we define $\Ths[T]$ as the set of all simplices of $\Ths$  contained in $T$ (see Figure \ref{fig:simplices.faces.T.a}) and $\Fhsi[T]$ as the set of faces of $\Fhs$ that lie in the interior of $T$.
For $F \in \Fh$, $\Fhs[F]$ denotes the set of simplicial faces $\sigma$ for which $\sigma \subset F$, and we let $\bn_\sigma\coloneq\bn_{TF}$, and $\bn_{\tau\sigma}\coloneq\bn_{\sigma}$ for the unique element $\tau\in \Ths[T]$, $T\in\Th$, which contains $\sigma$. 
{Additional notations for mesh elements and faces are introduced at the beginning of Section \ref{sec:discrete.setting:gradient.reconstruction} and illustrated in Figure \ref{fig:simplices.faces.T.b}}.
For future use, we notice that, {by \cite[Lemma 1.12]{Di-Pietro.Droniou:20}}, mesh regularity implies the existence of an integer $N\geq 0$ depending only on the mesh regularity parameter such that
\begin{align}
\max_h \max_{T\in\Th} \card(\Ths[T]) \leq N
\qquad \text{and}\qquad
\max_h \max_{T\in\Th} \card(\Fh[T]) \leq N.
\label{ineq:card.IT.F}
\end{align}
%%%%%%%%%%%%%%%%%%%%%%%%%%%%%%%%%%%%%%%%%%%%%%%%%%%%%%%%%%%%
%% Figure for notation
\begin{figure}[ht]
  \begin{minipage}[t]{.5\textwidth}
    \centering
    % include first image
    \def\svgwidth{.8\columnwidth}
    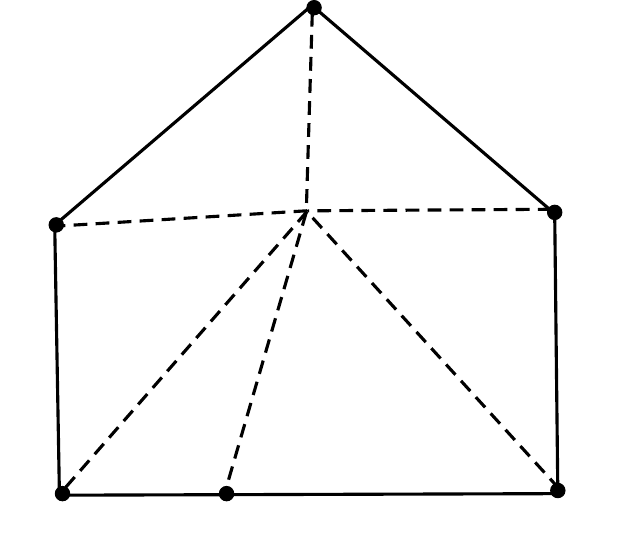
    \subcaption{The elements of $\Ths[T]$ and $\Fh[T]$.}
    \label{fig:simplices.faces.T.a}
  \end{minipage}
  \begin{minipage}[t]{.5\textwidth}
    \centering
    % include second image
    \def\svgwidth{0.88\columnwidth}
    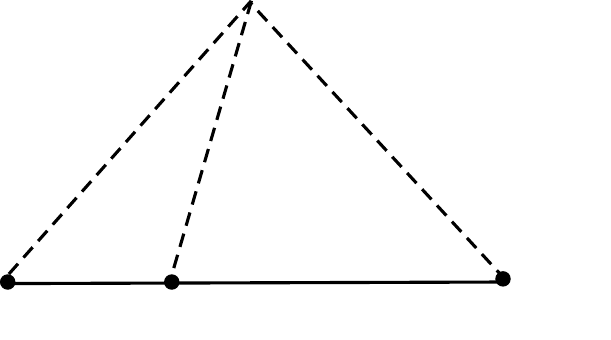
    \subcaption{A closer look to the bottom part: The faces $\sigma_1,\sigma_2,\sigma_3$ are interior faces, i.e., $\{\sigma_1,\sigma_2,\sigma_3\}\subset\Fhsi[T]$. 
      For the set $\{\sigma_4,\sigma_5\}$, we have $\sigma_4=F_1$ and $\sigma_5=F_2$.}
    \label{fig:simplices.faces.T.b}
  \end{minipage}
  \caption{An illustration of the sets $\Ths[T], \Fh[T]$ and $\Fhsi[T]$ for a given element $T\in \Th$ in $\Real^2$.}
  \label{fig:simplices.faces.T}
\end{figure}

%%%%%%%%%%%%%%%%%%%%%%%%%%%%%%%%%%%%%%%%%%%%%%%%%%%%%%%%%%%%
%% Figure for meshes
\begin{figure}[ht]
  \begin{minipage}[t]{.5\textwidth}
    \centering
    % include first image
    \def\svgwidth{.66\columnwidth}
    %% Creator: Inkscape inkscape 0.92.5, www.inkscape.org
%% Accompanies image file 'pentagon_cell_div_xT_pyramidal.pdf' (pdf, eps, ps)
\begingroup%
  \makeatletter%
  \providecommand\color[2][]{%
    \errmessage{(Inkscape) Color is used for the text in Inkscape, but the package 'color.sty' is not loaded}%
    \renewcommand\color[2][]{}%
  }%
  \providecommand\transparent[1]{%
    \errmessage{(Inkscape) Transparency is used (non-zero) for the text in Inkscape, but the package 'transparent.sty' is not loaded}%
    \renewcommand\transparent[1]{}%
  }%
  \providecommand\rotatebox[2]{#2}%
  \newcommand*\fsize{\dimexpr\f@size pt\relax}%
  \newcommand*\lineheight[1]{\fontsize{\fsize}{#1\fsize}\selectfont}%
  \ifx\svgwidth\undefined%
    \setlength{\unitlength}{248.14182687bp}%
    \ifx\svgscale\undefined%
      \relax%
    \else%
      \setlength{\unitlength}{\unitlength * \real{\svgscale}}%
    \fi%
  \else%
    \setlength{\unitlength}{\svgwidth}%
  \fi%
  \global\let\svgwidth\undefined%
  \global\let\svgscale\undefined%
  \makeatother%
  \begin{picture}(1,0.96977535)%
    \lineheight{1}%
    \setlength\tabcolsep{0pt}%
    \put(0,0){\includegraphics[width=\unitlength,page=1]{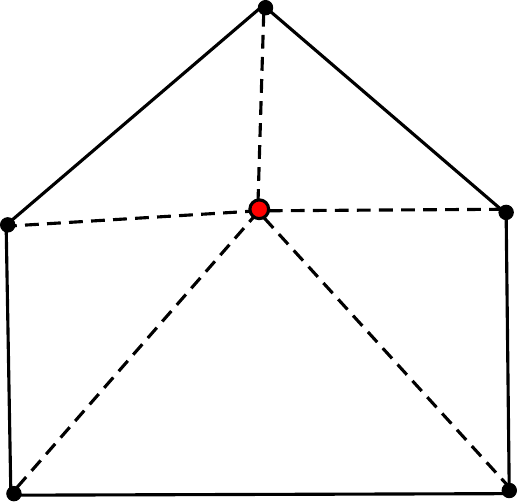}}%
    \put(0.52266849,0.58311779){\makebox(0,0)[lt]{\lineheight{1.25}\smash{\begin{tabular}[t]{l}$\boldsymbol{x}_T$\end{tabular}}}}%
  \end{picture}%
\endgroup%

    \subcaption{Pyramidal submesh.}
    \label{fig:ex.mesh.a}
  \end{minipage}
  \begin{minipage}[t]{.5\textwidth}
    \centering
    % include first image
    \def\svgwidth{.66\columnwidth}
    %% Creator: Inkscape inkscape 0.92.5, www.inkscape.org
%% Accompanies image file 'pentagon_cell_div_xT_Non_pyramidal.pdf' (pdf, eps, ps)
\begingroup%
  \makeatletter%
  \providecommand\color[2][]{%
    \errmessage{(Inkscape) Color is used for the text in Inkscape, but the package 'color.sty' is not loaded}%
    \renewcommand\color[2][]{}%
  }%
  \providecommand\transparent[1]{%
    \errmessage{(Inkscape) Transparency is used (non-zero) for the text in Inkscape, but the package 'transparent.sty' is not loaded}%
    \renewcommand\transparent[1]{}%
  }%
  \providecommand\rotatebox[2]{#2}%
  \newcommand*\fsize{\dimexpr\f@size pt\relax}%
  \newcommand*\lineheight[1]{\fontsize{\fsize}{#1\fsize}\selectfont}%
  \ifx\svgwidth\undefined%
    \setlength{\unitlength}{258.69884312bp}%
    \ifx\svgscale\undefined%
      \relax%
    \else%
      \setlength{\unitlength}{\unitlength * \real{\svgscale}}%
    \fi%
  \else%
    \setlength{\unitlength}{\svgwidth}%
  \fi%
  \global\let\svgwidth\undefined%
  \global\let\svgscale\undefined%
  \makeatother%
  \begin{picture}(1,0.99949605)%
    \lineheight{1}%
    \setlength\tabcolsep{0pt}%
    \put(0,0){\includegraphics[width=\unitlength,page=1]{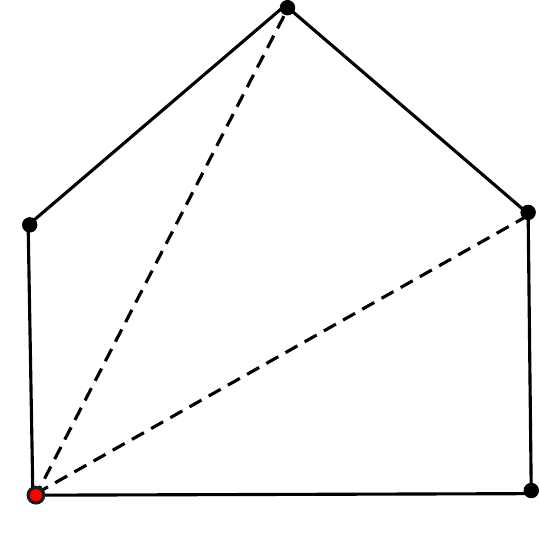}}%
    \put(-0.0028878,0.00820475){\makebox(0,0)[lt]{\lineheight{1.25}\smash{\begin{tabular}[t]{l}$\boldsymbol{x}_T$\end{tabular}}}}%
  \end{picture}%
\endgroup%

    \subcaption{Non-pyramidal submesh.}
    \label{fig:ex.mesh.b}
  \end{minipage}
  \caption{Two examples of submeshes $\Ths[T]$ in $\Real^2$ that satisfy the assumptions of the Section \ref{sec:setting:mesh} . The red dot represents the vertex $\bx_T$.}
  \label{fig:ex.mesh.T}
\end{figure}
We additionally make the assumption that, for all element $T\in\Th$, its submesh $\Ths[T]$ is constructed in such way that all simplices in $\Ths[T]$ have at least one common vertex
{(see Remarks \ref{rem:cvertex.prop.low} and \ref{rem:cvertex.prop}  for the technical details of this assumption)}.
  This vertex will be denoted ${\bx}_T$. In particular, when $\bx_T$ lies in the interior of $T$, we call  $\Ths[T]$ a  pyramidal submesh. The Figure \ref{fig:ex.mesh.T} shows two examples of submeshes that satisfy the current assumption.

In order to prevent the proliferation of generic constants we write, whenever possible, $a\lesssim b$ in place of $a\le Cb$ with $C>0$ independent of $\nu$, $\lambda$, $h$ and, for local inequalities, also on the mesh element or face.
The dependencies of the hidden constant will be further specified when relevant. Moreover, we write
$a  \simeq b$, when both $a \lesssim b$  and $b \lesssim a$ hold.

\subsection{Local and broken spaces and projectors}

Let $X$ denote a mesh element or face and, for an integer $l\geq 0$, denote by
 $\Poly{l}(X)$  the space spanned by the restrictions to $X$
of polynomials in the space variables of total degree $\leq l$.
The $L^2$-orthogonal projector $\pi_X^l:L^1(X)\rightarrow \Poly{l}(X)$ is such that, for all $\zeta\in L^1(X)$,
\begin{equation}
  \int_X (\zeta-\pi_X^l \zeta)w=0 \qquad \forall w \in \Poly{l}(X).\label{eq:l2proj:def}
\end{equation}
Vector and matrix versions of the $L^2$-orthogonal projector are obtained by applying $\pi_X^l$ component-wise, and are both denoted with the bold symbol $\bpi_X^l$ in what follows.
Optimal approximation properties for the $L^2$-orthogonal projector are proved in \cite[Appendix A.2]{Di-Pietro.Droniou:17}; see also \cite[Chapter 1]{Di-Pietro.Droniou:20}, where these estimates are extended to non-star shaped elements.
Specifically, let $s\in\{0,\dots,l+1\}$ and $r \in [1,+\infty]$.
Then, it holds, with hidden constant only depending on $l$, $s$, $r$, and the mesh regularity parameter:
For all $T\in {\calT_h}$, all $\zeta\in W^{s,r}(T)$, and all $m\in\{0,\dots,s\}$,
\begin{subequations}
  \begin{equation}
    |\zeta - \pi_T^l \zeta |_{W^{m,r}(T)} \lesssim h_T^{s-m}|\zeta|_{W^{s,r}(T)},\label{eq:l2proj:error:cell}
  \end{equation}
  and, if $s \geq 1$ and $m\leq s-1$,
  \begin{equation}
    h_T^{\frac{1}{r}}|\zeta - \pi_T^l \zeta |_{W^{m,r}(\calF_T)} \lesssim h_T^{s-m}|\zeta|_{W^{s,r}(T)},\label{eq:l2proj:error:faces}
  \end{equation}
\end{subequations}
where $W^{m,r}(\calF_T)$ is the space spanned by functions in $L^r(\partial T)$ that are in $W^{m,r}(F)$ for all $F\in\calF_T$,
endowed with the corresponding broken norm. 

At the global level, the space of broken polynomial functions on $\calT_h$ of total degree $\leq l$ is denoted by $\Poly{l}(\calT_h)$, and $\pi_h^l$ is the corresponding $L^2$-orthogonal projector such that, for all $\zeta \in L^1(\Omega)$, $(\pi_h^l \zeta)_{|T}\coloneq\pi_T^l \zeta_{|T}$ for all $T\in\calT_h$.
Regularity requirements in error estimates will be expressed  in terms of the broken Sobolev spaces $W^{s,r}(\calT_h)$ spanned by functions in $L^r(\Omega)$ the restriction of which to every $T\in\Th$ is in $W^{s,r}(T)$.
We additionally set, as usual, $H^s(\calT_h)\coloneq W^{s,2}(\calT_h)$.

\subsection{Discrete spaces and norms}
\label{sec:discspaces}
Let a polynomial degree $k\geq 0$ be fixed.
We define the HHO space as usual, setting
\begin{align*}
  \uline{\bU}_h^k\coloneq \left\{
  \uline{\bv}_h=((\bv_T)_{T\in \calT_h},(\bv_F)_{F\in \calF_h}) : 
  \mbox{$\bv_T \in \Poly{k}(T)^3$ for all $T \in \calT_h$ and
    $\bv_F \in \Poly{k}(F)^3$ for all $F \in \calF_h$}
  \right\}.  
\end{align*}
The restrictions of $\uline{\bU}_h^k$ and $\uline{\bv}_h \in \uline{\bU}_h^k$ to a generic mesh element $T\in \calT_h$ are respectively denoted by $\uline{\bU}_T^k$ and $\uline{\bv}_T=(\bv_T,(\bv_F)_{F\in \calF_T})$.
The vector of polynomials corresponding to a smooth function over $\Omega$ is obtained via the global interpolation operator $\uline{\bI}_h^k: H^1(\Omega)^3 \rightarrow \uline{\bU}_h^k$ such that, for all $\bv \in H^1(\Omega)^3$,
\begin{equation}\label{eq:Ih}
  \uline{\bI}_h^k\bv \coloneq  ((\bpi_T^k\bv_{|T})_{T\in\calT_h},(\bpi_F^k\bv_{|F})_{F\in \calF_h}).
\end{equation}
Its restriction to a generic mesh element $T \in \calT_h$, collecting the components on $T$ and its faces, is denoted by $\uline{\bI}_T^k$.
We furnish $\uline{\bU}_{h}^k$ with the discrete $H^1$-like seminorm such that, for all
$\uline{\bv}_h \in \uline{\bU}_h^k$,
\begin{equation*}%% \label{eq:norm.1h}
  \|\underline{\bv}_h\|_{1,h}\coloneq\left(
  \sum_{T \in \calT_h} \|\underline{\bv}_T\|_{1,T}^2
  \right)^{\nicefrac12},
\end{equation*}
where, for all $T\in\calT_h$,
\begin{equation}\label{eq:norm.1T}
  \|\underline{\bv}_T\|_{1,T}\coloneq\left( 
  \|\nabla {\bv}_T\|_{L^2(T)^{3\times 3}}^2 + |\uline{\bv}_T|_{1,\partial T}^2
  \right)^{\nicefrac12}
  \mbox{ with }
  |\uline{\bv}_T|_{1,\partial T}\coloneq\left(
  \sum_{F\in \calF_T}\! h_F^{-1}\| \bv_F-\bv_T\|^2_{L^2(F)^3}
  \right)^{\nicefrac12}.
\end{equation}

The discrete spaces for the velocity and the pressure, respectively accounting for the wall boundary condition and the zero-average condition, are 
\begin{equation*}%% \label{eq:ns:vUhD:Ph}
  \uline{\bU}_{h,0}^k
  \coloneq
  \left \{
  \uline{\bv}_h = ((\bv_T)_{T\in\calT_h},(\bv_F)_{F\in\calF_h})\in  \uline{\bU}_h^k : \bv_F = \boldsymbol{0} \quad \forall F\in\calF_h^{\rm b}
  \right\},\qquad
  P_h^k \coloneq \Poly{k}({\cal {T}}_h)\cap P.
\end{equation*}
For all  $\uline{\bv}_h \in \uline{\bU}_h^k$, we denote by ${\bv}_h \in \Poly{k}(\calT_h)^3$ the vector-valued broken polynomial function obtained patching element-based unkowns, that is $(\bv_h)_{|T}\coloneq \bv_T$ for all $T \in \calT_h$.
The following discrete Sobolev embeddings in $\uline{\bU}_{h,0}^k$ have been proved in \cite[Proposition 5.4]{Di-Pietro.Droniou:17}:
For all $r\in[1,6]$ it holds, for all  $\uline{\bv}_h \in \uline{\bU}_{h,0}^k$,
\begin{align}
  \|{\bv}_h\|_{L^r(\Omega)^3}
  \lesssim
  \|\uline{\bv}_h\|_{1,h}.
  \label{eq:disc:sobembd}
\end{align}
where the hidden constant is independent of both $h$ and $\uline{\bv}_h$, but possibly depends on $\Omega$, $k$, $r$, and the mesh regularity parameter.
It follows from \eqref{eq:disc:sobembd} that the map $\| {\cdot} \|_{1,h}$ defines a norm on $\uline{\bU}_{h,0}^k$.
Classically, the corresponding dual norm of a linear form $\mathcal{L}_h:\uline{\bU}_{h,0}^k\to\Real$ is given by
\begin{equation}\label{eq:dual.norm}
  \|\mathcal{L}_h\|_{1,h,*}
  \coloneq\sup_{\uline{\bv}_h\in\uline{\bU}_{h,0}^k,\|\uline{\bv}_h\|_{1,h}=1}\left|
  \mathcal{L}_h(\uline{\bv}_h)
  \right|.
\end{equation}

%------------------------------------------------------------------------------%

\subsection{Divergence-preserving local velocity reconstruction}\label{sec:discrete.setting:velocity.reconstruction}

Following  \cite{Di-Pietro.Ern:15}, for any element $T \in \calT_h$ we define the discrete divergence operator $D_T^k:\uline{\bU}_T^k \rightarrow \Poly{k}(T)$ such that, for all $\uline{\bv}_T \in \uline{\bU}_T^k$ and all $q \in \Poly{k}(T)$,
\begin{equation}\label{eq:hho:div:op}
  \int_T D_T^k\uline{\bv}_T q
  = - \int_T \bv_T\cdot \GRAD q + \sum_{F\in \calF_T} \int_F (\bv_F \cdot \bn_{TF}) q.
\end{equation}
Crucially, the operator $D_T^k$ satisfies the following commutation property (see \cite[Eq. (8.21)]{Di-Pietro.Droniou:20}):
\begin{equation}
  D_T^k\uline{\bI}_T^k \bv  = \pi_T^k(\DIV \bv) \qquad \forall \bv \in H^1(T)^3.
  \label{eq:DT.commuting}
\end{equation}

To achieve pressure robustness in the sense made precise by Remark \ref{rem:pressure-robustness} below, we reconstruct divergence-preserving velocity test functions, which are used for the discretization of the body force and the nonlinear term.
Let an element $T \in \calT_h$ be fixed and, for $\tau \in \Ths[T]$, denote by $\RTN{k}(\tau)\coloneq\Poly{k}(\tau)^3+\bx\Poly{k}(\tau)$ the local Raviart--Thomas--N\'ed\'elec space of degree $k$ \cite{Raviart.Thomas:77,Nedelec:80}.
We recall that a function in $\RTN{k}(\tau)$ is uniquely determined by its polynomial moments of degree up to $(k-1)$ inside $\tau$ and the polynomial moments of degree $k$ of its normal component on each face $\sigma\in\Fhs[\tau]$ (with $\Fhs[\tau]$ denoting the subset of $\Fhs$ collecting the simplicial faces of $\tau$).
We additionally note the following local norm equivalence uniform in $h$:
\begin{align}
  \norm{\Ldeux[\tau]^3}{\bws}^2
  \simeq
  \norm{\Ldeux[\tau]^3}{\bpi_\tau^{k-1}\bws}^2
  +
  \sum_{\sigma\in \Fhs[\tau]}h_\sigma \norm{\Ldeux[\sigma]}{\bws \cdot\bn_{\tau\sigma}}^2
    \qquad \forall   \bws \in \RTN{k}(\tau).
  \label{eq:RTN.L2norm}
\end{align}
We introduce the Raviart--Thomas--N\'ed\'elec space of degree $k$ on the matching simplicial submesh $\Ths[T]$ of $T$ defined as follows:
\begin{equation*}%% \label{eq:def:rtn:gbl}
  \RTN{k}(\Ths[T])\coloneq\left\{
  \bws \in \Hdiv[T] : \text{%
    $\bws_{|\tau}\in \RTN{k}(\tau)$ for all $\tau \in \Ths[T]$
  }
  \right\},
\end{equation*}
where $\Hdiv[T]\coloneq\{\bws \in L^2(T)^3:  \DIV \bws \in \Ldeux[T]\}$.
We also introduce the subspace of $\RTN{k}(\Ths[T])$ spanned by functions with zero normal trace on the boundary of $T$:
\begin{equation*}%% \label{eq:def:rtn0:gbl}
  \RTN{k}_0(\Ths[T])
  \coloneq\left\{
  \bws \in   \RTN{k}(\Ths[T]) :
  \bws \cdot \text{$\bn_{\sigma} = 0$ for all $\sigma  \in \Fhs[F]$ and all $F \in \Fh[T]$}
  \right\}.
\end{equation*}

Recall from Section \ref{sec:setting:mesh} that, for a given element $T\in\Th$,  we denote by $\bx_T$ the common vertex of all simplices in $\Ths[T]$.
With this in mind, we additionally introduce the following space generated by the Koszul operator (\cite[Section 7.2]{Arnold:18}):
\begin{align*}
\Goly{{\rm c},k}(T)\coloneq (\bx-\bx_T) \times \Poly{k-1}(T)^3 \qquad \text{for } k\geq1,
\end{align*}
and define $\Goly{{\rm c},-1}(T)\coloneq\Goly{{\rm c},0}(T)\coloneq\{0\}$.
Observe that we have the following
decomposition for $\Poly{k}(T)^3$ (see \cite[Corollary 7.4]{Arnold:18}):
\begin{equation}\label{eq:kz.decomp}
\Poly{k}(T)^3=\GRAD \Poly{k+1}(T) \oplus \Goly{{\rm c},k}(T),
\end{equation}
where the direct sum above is not orthogonal in general.
Additionally, we define the $L^2$-orthogonal projector on the space $\Goly{{\rm c},k}(T)$  as $\bpi^{ {\rm c},k}_{\Goly{},T}$.
Then, the \emph{divergence-preserving velocity reconstruction} ${\bR}_{T}^k: \uline{\bU}_T^k \rightarrow \RTN{k}(\Ths[T])$ is defined, for all $\uline{\bv}_T \in \uline{\bU}_T^k$, as the first component of the solution of the following mixed problem: 
Find  $({\bR}_{T}^k\uline{\bv}_T,\ps,\ts)\in \RTN{k}(\Ths[T])\times \Poly{k}(\Ths[T])\times\Goly{{\rm c},k-1}(T)$ such that
\begin{subequations}
  \label{eq:darcyT:weak}
  \begin{alignat}{2}
    ({\bR}_{T}^k\uline{\bv}_T)_{|\sigma} \cdot \bn_{\sigma}&= (\bv_F\cdot \bn_{TF})_{|\sigma}
    &\qquad&
    \forall \sigma  \in \Fhs[F],\,  \forall F \in \Fh[T],
    \label{eq:darcyT:weak:bd}\\
    \int_{T} (\DIV {\bR}_{T}^k\uline{\bv}_T) \qs &= \int_{T} (D_T^k\uline{\bv}_T) \qs
    &\qquad&
    \forall \qs \in \Poly{k}(\Ths[T]),
    \label{eq:darcyT:weak:a}\\
    \int_{T}  {\bR}_{T}^k\uline{\bv}_T \cdot \xis &= 
    \int_{T} \bv_T\cdot \xis
   &\qquad&
    \forall \xis \in \Goly{{\rm c},k-1}(T),
    \label{eq:darcyT:weak:ab}\\
    \int_{T}  {\bR}_{T}^k\uline{\bv}_T \cdot \bws  +  \int_{T} (\DIV \bws )\ps 
    + \int_{T}  {\bws} \cdot \ts
    &= \int_{T} \bv_T \cdot \bws
    &\qquad&
    \forall \bws \in\RTN{k}_0(\Ths[T]).
    \label{eq:darcyT:weak:b}
  \end{alignat}
\end{subequations}

\begin{remark}[Allowing more than pyramidal meshes]
A similar divergence-preserving operator has been proposed in 
\cite[Section 4.2]{Lederer:2017} in the context of finite elements pairs with continuous pressures. 
However, adapting it to the current  HHO framework will restrict the submesh $\Ths[T]$  to be only a pyramidal submesh (or  a vertex patch in the terminology of \cite{Lederer:2017}). Specifically,
using the methodology introduced in \cite[Proof of Theorem 12]{Lederer:2017},
 to prove the equation \eqref{ineq:rtn:consis} below, 
it will be necessary to construct the  Lagrange hat function of $\bx_T$ (a polynomial function $q\in\Poly{1}(\Ths[T])$ 
such that $q(\bx_T)=1$ and vanishes at the other vertices of 
$\Ths[T]$)
and use its properties with the crucial restriction
that this hat function must vanish at the boundary of $T$.
This is only possible  when $\Ths[T]$ is pyramidal.
In the current manuscript, we avoid this restriction using Lemma \ref{lemma:appx} below.
\end{remark}

\begin{lemma}[Properties of $\bR_{T}^k$]
  \label{lemm:rtn}
  It holds:
  \begin{enumerate}[(i)]
  \item \emph{Well-posedness.}
    For a given $\uline{\bv}_T \in \uline{\bU}_T^k $, there exists a unique solution to problem  \eqref{eq:darcyT:weak},
    and it holds that  
    \begin{align}
      \norm{\Ldeuxd[T]}{\bv_T - \bR_{T}^k\uline{\bv}_T  } \lesssim h_T   \seminorm{1,\partial T}{\uline{\bv}_T}.
      \label{ineq:rtn:bound}
    \end{align}
  \item \emph{Approximation.} For all $\bv\in \bH^{k+1}(T)^3$, it holds
    \begin{align}
      \norm{\Ldeuxd[T]}{\bv  - \bR_{T}^k(\uline{\bI}_T^k\bv)  } \lesssim h_T^{k+1}   \seminorm{H^{k+1}(T)^3}{\bv}.
      \label{ineq:rtn:approx}
    \end{align}
 \item \emph{Consistency.} {For a given $\uline{\bv}_T \in \uline{\bU}_T^k$, it holds,
for $k\geq1$,
    \begin{align}
      \bpi^{k-1}_T(\bR_{T}^k\uline{\bv}_T)= \bpi^{k-1}_T(\bv_T).
      \label{ineq:rtn:consis}
    \end{align}
}
  \end{enumerate}
\end{lemma}
The proof makes use of the following Lemma, whose proof is given in Appendix \ref{sec:appx}.

\begin{lemma}[Raviart--Thomas lifting of the projection in $\Goly{{\rm c},k-1}(T)$]\label{lemma:appx}
  Let $T\in \Th$ and a function $\bv\in L^2(T)^3$ be given.
  Then, for $k\geq2$, there exists $\widetilde{\bR}_T^k(\bv)\in\RTN{k}_0(\Ths[T])$ such that
 \begin{subequations}
   \label{eq:rtn:golyc}
   \begin{align}
   \bpi^{{\rm c},k-1}_{\Goly{},T}\widetilde{\bR}_T^k(\bv)&=\bpi^{{\rm c},k-1}_{\Goly{},T}\bv,
    \label{eq:rtn:golyc:a}\\
   \DIV \widetilde{\bR}_T^k(\bv)&=0,
    \label{eq:rtn:golyc:b}\\
%%     \forall \sigma\in \Fhs[T]^{\rm i},
%%     \qquad
%%     \int _\sigma (\widetilde{\bR}_T^k(\bv) \cdot\bn_{\sigma})\qs 
%%     &=0
%%     \qquad\forall  \qs \in \Poly{k}(\sigma),
    \widetilde{\bR}_T^k(\bv) \cdot\bn_{\sigma} &= 0
    \qquad\forall \sigma\in \Fhs[T]^{\rm i},
   \label{eq:rtn:golyc:c}\\
   \norm{\Ldeuxd[T]}{\widetilde{\bR}_T^k(\bv)}&\lesssim\norm{\Ldeuxd[T]}{\bv}
    \label{eq:rtn:golyc:d}
    \end{align}
 \end{subequations}
\end{lemma}

\begin{remark}[The common vertex assumption for $k\in\{0,1\}$]\label{rem:cvertex.prop.low}
  The common vertex assumption described in Section \ref{sec:setting:mesh} is not necessary for $k\in\{0,1\}$ since, for those cases, $\Goly{{\rm c},k-1}$ becomes the trivial space, and Lemma \eqref{lemma:appx} is not needed.
\end{remark}

\begin{proof}
  \emph{(i) Well-posedness.} We prove this item in three parts starting with the existence and uniqueness of ${\bR}_{T}^k\uline{\bv}_T$.
  \smallskip\\
  \underline{(i.A) \emph{Existence, uniqueness, and decomposition of  ${\bR}_{T}^k\uline{\bv}_T$.}}
  The existence and uniqueness of a solution to problem  \eqref{eq:darcyT:weak} follows from the classical theory of mixed problems given the compatibility of the selected  spaces; see, e.g., \cite[Section 14]{Ciarlet.Lions:91} and  Lemma \ref{lemma:appx} below.
  In order to prove the a priori estimate \eqref{ineq:rtn:bound}, we decompose ${\bR}_{T}^k\uline{\bv}_T$ as follows:
  \begin{equation}
    \label{eq:v.des}
    \bR_{T}^k\uline{\bv}_T =  {\bvs}' +  \bvs_0,
  \end{equation}
  where:
  \begin{itemize}
  \item ${\bvs}' \in  \RTN{k}(\Ths[T])$ is a lifting of the boundary values defined by prescribing its DOFs as follows:
    \begin{subequations}
      \label{eq:vp.idofs}
      \begin{alignat}{4}
        \label{eq:vp.dofs.int.T}
        &\forall \tau \in \Ths[T],
        &\quad
        \int_\tau \bvs' \cdot \bws
        &= 0,
        &\qquad&  \forall \bws \in \Poly{k-1}(\tau)^3,
        \\ \label{eq:vp.dofs.int}
        &\forall \sigma\in \Fhs[T]^{\rm i},
        &\quad
        \int _\sigma (\bvs' \cdot\bn_{\sigma})\qs 
        &=0
        &\qquad& \forall  \qs \in \Poly{k}(\sigma),
        \\ \label{eq:vp.dofs.bd}
        &
        \forall F \in \Fh[T],\,\forall \sigma \in \Fhs[F],
        &\quad
        \bvs'_{|\sigma} \cdot \bn_{\sigma}
        &= (\bv_F\cdot \bn_{TF})_{|\sigma};
      \end{alignat}
    \end{subequations}
  \item Letting $\Poly{k}_0(\Ths[T])\coloneq \left \{ \qs \in \Poly{k}(\Ths[T]): \int_T\qs=0\right \}$,
    $\bvs_0$ is the first component of the unique solution to the following mixed problem: 
    Find $(\bvs_0,\ps,\ts)\in \RTN{k}_0(\Ths[T])\times \Poly{k}_{0}(\Ths[T])\times\Goly{{\rm c},k-1}(T)$ such that
    \begin{subequations}
      \label{eq:darcyT0:weak}
      \begin{alignat}{2}
        \int_{T} (\DIV \bvs_0) \qs &= \int_{T} (D_T^k\uline{\bv}_T - \DIV \bvs') \qs
        &\qquad&  \forall \qs \in \Poly{k}_0(\Ths[T]),
        \label{eq:darcyT0:weak:a}\\
        \int_{T}  {\bvs}_0 \cdot \xis &= 
        \int_{T} \bv_T\cdot \xis
        &\qquad&
        \forall \xis \in \Goly{{\rm c},k-1}(T),
        \label{eq:darcyT0:weak:ab}\\
        \int_{T}  \bvs_0 \cdot \bws + \int_{T} (\DIV \bws )\ps 
        + \int_{T}  {\bws} \cdot \ts
        &=   \int_{T} ( \bv_T  -\bvs') \cdot \bws
        &\qquad&
        \forall \bws \in \RTN{k}_0(\Ths[T]),
        \label{eq:darcyT0:weak:b}
      \end{alignat}
    \end{subequations}
    where we have used equation \eqref{eq:vp.dofs.int.T} along with the fact that $\Goly{{\rm c},k-1}(T)\subset\Poly{k-1}(\Ths[T])^3$ in the right hand side of \eqref{eq:darcyT:weak:ab} to write $\bv_T$ instead of $\bv_T - \bvs'$.
  \end{itemize}
  \noindent\underline{(i.B) \emph{Boundedness}}.
  We begin by proving the following estimate:
  \begin{align}
    \label{eq:bound.solddarcy}
    \norm{\Ldeuxd[T]}{{\bR}_{T}^k\uline{\bv}_T}
    &
    \lesssim
    \norm{\Ldeuxd[T]}{{\bv}_T}
    +
    h_T
    \|\uline{\bv}_T\|_{1,T}
    +
    \sum_{F\in \Fh[T]}h_F^{\frac{1}{2}} \norm{\Ldeuxd[F]}{\bv_F}.   
  \end{align}
  Let $\Ldeuxxz[T]\coloneq \left \{ \xi \in \Ldeuxx[T]: \int_T\xi=0 \right \}$ and denote by $\theta \in \Hun[T]\cap \Ldeuxxz[T]$ the solution of the equation
  \begin{align}
    \label{eq:neumann:weak}
    \int_T \GRAD \theta \cdot \GRAD \xi = \int_T (D_T^k\uline{\bv}_T  - \DIV \bvs')  \xi  \qquad  \forall \xi \in \Hun[T]\cap \Ldeuxxz[T].
  \end{align}
  We recall that \eqref{eq:neumann:weak} is the weak form of the following strong Neumann problem
  \begin{subequations}
    \label{eq:neumann:strong}
    \begin{alignat}{2}
      - \LAP \theta  &= (D_T^k\uline{\bv}_T  - \DIV \bvs')  &\quad&   \text{in } T,
      \label{eq:neumann:strong:pde}\\
      \frac{\partial \theta}{\partial n}&=0 &\quad&   \text{on } \partial T.
      \label{eq:neumann:strong:bd}
    \end{alignat}
  \end{subequations}
  Since 
    $\int_T (D_T^k\uline{\bv}_T  - \DIV \bvs') = 0$ by  \eqref{eq:hho:div:op} with $q=1$ and using integration by parts for the integral having $\DIV \bvs'$  along with \eqref{eq:vp.dofs.bd}, the compatibility condition for problem   \eqref{eq:neumann:strong} is satisfied, yielding existence and uniqueness of $\theta$.
  Moreover, since $T$ is  a convex polyhedron and the forcing term is in $\Ldeux[T]$,  then $\theta \in H^2(T)\cap \Ldeuxxz[T]$ (see \cite[Section 8.2]{Grisvard:11}) with
  \begin{align}
    \label{ineq:bound:thetH2} 
    \seminorm{H^2(T)}{\theta}
    \lesssim
    %  \norm{\Ldeuxd[T]}{ \LAP \theta}=
    \norm{\Ldeux[T]}{D_T^k\uline{\bv}_T  - \DIV \bvs'},
  \end{align}
  where it can be checked following the argument in the reference that the hidden constant does not depend on  $T$.
  Setting $\xi=\theta$ in \eqref{eq:neumann:weak} and using the Cauchy--Schwarz inequality followed by the Poincaré--Wirtinger inequality $    \norm{\Ldeux[T]}{\zeta} \leq \frac{h_T}{\pi}\seminorm{\Hun[T]}{\zeta}$ valid for all $\zeta \in H^1(T)\cap L^2_0(\Omega)$ (see \cite{Payne.Weinberger:60,Bebendorf:03}), we estimate
  \begin{align}
    \label{ineq:bound:thetH1}
    \seminorm{\Hun[T]}{\theta}
    \lesssim  {h_T}\norm{\Ldeux[T]}{D_T^k\uline{\bv}_T  - \DIV \bvs'}.
  \end{align}
  Let now $\bz_0\coloneq-\GRAD \theta\in\Hund[T]$.
  Defining $\hat{\bz}_0$ $\in\RTN{k}_0(\Ths[T])$ as the interpolate of $\bz_0$ onto $\RTN{k}(\Ths[T])$, and using the commutation property $\DIV \hat{\bz}_0=\pi_{\Ths[T]}^k(\DIV \bz_0)$ with $\pi_{\Ths[T]}^k$ denoting the $L^2$-orthogonal projector onto  $\Poly{k}(\Ths[T])$ (see, e.g., \cite[Section 2.5.2]{Boffi.Brezzi.ea:13}), it is inferred that
  $\DIV \hat{\bz}_0 =\pi_{\Ths[T]}^k(\DIV \bz_0) = (D_T^k\uline{\bv}_T  - \DIV \bvs')$,
  where we have used \eqref{eq:neumann:strong:pde} in the last step.
 Therefore, by \eqref{eq:darcyT0:weak:a}, $\DIV ( \bvs_0- {\hat{\bz}_0}) = 0$.
 Now, using Lemma \ref{lemma:appx}, let 
  $\hat{\bz}_1\coloneq\widetilde{\bR}_T^k(\bv_T - \hat{\bz}_0)\in \RTN{k}_0(\Ths[T])$,
  and set $\hat{\bz}\coloneq\hat{\bz}_0 +\hat{\bz}_1$.
  Using \eqref{eq:rtn:golyc:b}, we get $\DIV( \bvs_0 - \hat{\bz}) =0$; 
  moreover, using \eqref{eq:rtn:golyc:a} and \eqref{eq:darcyT0:weak:ab}, we have $\boldsymbol{\pi}^{{\rm c},k-1}_{\boldsymbol{\cal G},T}(\bvs_0 - \hat{\bz}) = \boldsymbol{0}$.
  Taking then $\bws = \bvs_0 - \hat{\bz}$ as a test function in \eqref{eq:darcyT0:weak:b}, we obtain
  \begin{equation}
    \label{eq:ortho:prop}
    \int_T \bvs_0 \cdot(\bvs_0- \hat{\bz})
    = \int_{T} ( \bv_T  -\bvs') \cdot (\bvs_0- \hat{\bz}).
  \end{equation}
  Thus, it is readily seen that
  \begin{align}
    \norm{\Ldeuxd[T]}{ \bvs_0 - \hat{\bz}}^2
    &=
    \int_{T} (\bvs_0 - \hat{\bz}) \cdot( \bvs_0 - \hat{\bz})\notag\\
    &=
    \int_{T}  ( \bv_T  -\bvs' - \hat{\bz} ) \cdot( \bvs_0 - \hat{\bz})\notag\\
    &  \leq
    \left( \norm{\Ldeuxd[T]}{ \bv_T - \hat{\bz}}
    +
    \norm{\Ldeuxd[T]}{   \bvs' }
    \right)
    \norm{\Ldeuxd[T]}{  \bvs_0 - \hat{\bz}}\notag\\
    &  \leq
    \left( \norm{\Ldeuxd[T]}{ \bv_T - \hat{\bz}_0}
    +
    \norm{\Ldeuxd[T]}{ \hat{\bz}_1}
    +
    \norm{\Ldeuxd[T]}{   \bvs' }
    \right)
    \norm{\Ldeuxd[T]}{  \bvs_0 - \hat{\bz} }
    \notag\\
    &  \lesssim
    \left( \norm{\Ldeuxd[T]}{ \bv_T - \hat{\bz}_0}
    +
    \norm{\Ldeuxd[T]}{   \bvs' }
    \right)
    \norm{\Ldeuxd[T]}{  \bvs_0 - \hat{\bz}
    }\notag\\
    &  \leq
    \left(
    \norm{\Ldeuxd[T]}{ \bv_T - \bz_{0}}
    +
    \norm{\Ldeuxd[T]}{ \bz_{0} - \hat{\bz}_{0}}  
    +
    \norm{\Ldeuxd[T]}{   \bvs' }
    \right)
    \norm{\Ldeuxd[T]}{  \bvs_0 - \hat{\bz}}  \notag\\
    &  \eqcolon
    (\mathfrak{T}_1 + \mathfrak{T}_2 + \mathfrak{T}_3)
    \norm{\Ldeuxd[T]}{  \bvs_0 - \hat{\bz}},
    \label{eq:norm:v0:zhat}
  \end{align}
  where  we have used equation \eqref{eq:ortho:prop} in the second step, 
  Cauchy--Schwarz and triangle inequalities in the third step,
  the definition of $\hat{\bz}$ for the first term in the parentheses   and a triangle inequality in the fourth step,
   the definition of $\hat{\bz}_1$  along with the bound \eqref{eq:rtn:golyc:d} in the fifth step,
 and again a triangle inequality after inserting $\pm{\bz}_0$ into the first norm in the sixth step.\\
 To estimate $\mathfrak{T}_1$, we begin using the triangle inequality  to obtain
  \begin{align}
    \mathfrak{T}_1
    &\leq
    \norm{\Ldeuxd[T]}{\bv_T }
    +
    \norm{\Ldeuxd[T]}{ \bz_0}
    \lesssim
    \norm{\Ldeuxd[T]}{\bv_T }
    +
    h_T\norm{\Ldeux[T]}{D_T^k\uline{\bv}_T  - \DIV \bvs'}\notag\\
 &   \lesssim
    \norm{\Ldeuxd[T]}{\bv_T }
    +
     h_T\left (\norm{\Ldeux[T]}{D_T^k\uline{\bv}_T} + \norm{\Ldeux[T]}{\DIV \bvs'}\right),
\label{eq:RT-proof:T1:bound}
  \end{align}
where we have used \eqref{ineq:bound:thetH1} to bound the second term in the second step, and a triangle inequality in the last step. For the first term in parentheses, integrating by parts the right-hand side of \eqref{eq:hho:div:op}, applying Cauchy--Schwarz and discrete trace inequalities, and taking the supremum over $q\in\Poly{k}(T)$, we obtain
  \begin{align}
    \norm{\Ldeux[T]}{D_T^k\uline{\bv}_T}
    &\lesssim
    \|\uline{\bv}_T\|_{1,T}.
    \label{ineq:DT}
  \end{align}
  To bound the second term in parentheses, we use a discrete inverse inequality to write
  \begin{align}
    \norm{\Ldeux[T]}{\DIV \bvs'}^2
    \lesssim
    \sum _{\tau\in \Ths[T]}\seminorm{H^1(\tau)^3}{ \bvs'}^2
    \lesssim
    h_\tau^{-2}\sum _{\tau\in \Ths[T]}\norm{\Ldeuxd[\tau]}{\bvs'}^2
    \lesssim
    h_T^{-2}\norm{\Ldeuxd[T]}{\bvs'}^2,
    \label{ineq:vprime.div}
  \end{align}
  where the fact that $h_\tau^{-1}h_T\lesssim  1 $ for regular mesh sequences (see \cite[Eq. (1.4)]{Di-Pietro.Droniou:20}) has been used in the last step. 
  Now, to estimate $\norm{\Ldeux[T]^3}{\bvs'}$, we combine \eqref{eq:RTN.L2norm} and \eqref{eq:vp.idofs} to obtain 
  \begin{align}
    \label{ineq:vprime.l2}
    \norm{\Ldeux[T]^3}{\bvs'}^2
    \simeq
    \sum_{\sigma  \in \Fhs[F],F \in \Fh[T]}h_\sigma \norm{\Ldeux[\sigma]}{\bv_F\cdot \bn_{TF}}^2
    \leq
    \sum_{F\in \Fh[T]}h_F \norm{\Ldeuxd[F]}{\bv_F}^2,
  \end{align}
  where we have used the fact that $\sigma\subset F$,
  the inequality $h_\sigma\leq h_F$ valid for any $\sigma\in\Fhs[F]$ and any $F \in \Fh[T]$,
  and the H\"older inequality with exponents $(2,\infty)$ along with $\|\bn_{TF}\|_{L^\infty(F)^3}=1$
  for the third step.
  Therefore plugging  \eqref{ineq:vprime.l2} into \eqref{ineq:vprime.div}, we obtain
  \begin{align*}
    \norm{\Ldeux[T]}{\DIV \bvs'}^2
    \lesssim
    \sum_{F\in \Fh[T]}h_F^{-1} \norm{\Ldeuxd[F]}{\bv_F}^2,
  \end{align*}
  where we have  used the equivalence $h_F \simeq h_T$
  valid for regular mesh sequences (see \cite[Eq. (1.6)]{Di-Pietro.Droniou:20}). Now, using the following  bound  valid for non-negative real numbers $a_i$
  \begin{equation}\label{ineq:l2.bound}
        \sum_i a_i^2 \leq \left(\sum_i a_i \right)^2, 
  \end{equation}
for the previous inequality,  and then plugging the result along with \eqref{ineq:DT}  into
\eqref{eq:RT-proof:T1:bound}, it is inferred that
  \begin{equation}
    \label{eq:RT-proof:T1:bound:2}
    \mathfrak{T}_1
    \lesssim
    \norm{\Ldeuxd[T]}{\bv_T }
    +
    h_T
    \|\uline{\bv}_T\|_{1,T}
    +
    \sum_{F\in \Fh[T]}h_F^{\frac{1}{2}} \norm{\Ldeuxd[F]}{\bv_F}.
  \end{equation}
  \smallskip
  To bound the term $\mathfrak{T}_2$ in \eqref{eq:norm:v0:zhat}, we use  standard interpolation estimates for $\hat{\bz}_0$ (see, e.g., \cite[Proposition 2.5.4]{Boffi.Brezzi.ea:13}) followed by \eqref{ineq:bound:thetH2} to write
  \begin{equation}
    \mathfrak{T}_2
    \lesssim h_T\seminorm{H^2(T)}{\theta}
    \lesssim h_T \norm{\Ldeuxd[T]}{D_T^k\uline{\bv}_T  - \DIV \bvs'}
    \lesssim h_T\|\uline{\bv}_T\|_{1,T} + \sum_{F\in \Fh[T]}h_F^{\frac{1}{2}} \norm{\Ldeuxd[F]}{\bv_F},
    \label{eq:RT-proof:T2:bound}
  \end{equation}
  where we have used a triangle inequality followed by \eqref{ineq:DT}, \eqref{ineq:vprime.div}, and \eqref{ineq:vprime.l2} to conclude.
  \smallskip
  Plugging \eqref{eq:RT-proof:T1:bound:2} and \eqref{eq:RT-proof:T2:bound} into \eqref{eq:norm:v0:zhat}, using \eqref{ineq:vprime.l2} to estimate $\mathfrak{T}_3$, and simplifying, we obtain
  \begin{align}
    \label{eq:norm:v0:zhat:final}
    \norm{\Ldeuxd[T]}{ \bvs_0 - \hat{\bz}}
    &\lesssim
    \norm{\Ldeuxd[T]}{\bv_T }
    +
    h_T
    \|\uline{\bv}_T\|_{1,T}
    +
    \sum_{F\in \Fh[T]}h_F^{\frac{1}{2}} \norm{\Ldeuxd[F]}{\bv_F}.
  \end{align}
  Using the decomposition \eqref{eq:v.des} followed by  triangle inequalities, we finally get 
  \begin{align*}
    \norm{\Ldeuxd[T]}{{\bv}_T  -   {\bR}_{T}^k\uline{\bv}_T}
    %&\leq
    %\norm{\Ldeuxd[T]}{{\bv}_T  -  \bvs_0}
    %+   \norm{\Ldeuxd[T]}{ \bvs'}\\
    &\leq
    \norm{\Ldeuxd[T]}{{\bv}_T  -  \hat{\bz}}
    +
    \norm{\Ldeuxd[T]}{\hat{\bz} -  \bvs_0}
    +  \norm{\Ldeuxd[T]}{ \bvs'}\\
    &\lesssim
    \norm{\Ldeuxd[T]}{{\bv}_T  -  \hat{\bz}_0}
    +
    \norm{\Ldeuxd[T]}{\hat{\bz}_1}
    +
    \norm{\Ldeuxd[T]}{\hat{\bz} -  \bvs_0}
    +  \norm{\Ldeuxd[T]}{ \bvs'}
    \\
    &\lesssim
    \norm{\Ldeuxd[T]}{{\bv}_T  -  \hat{\bz}_0}
    +
    \norm{\Ldeuxd[T]}{\hat{\bz} -  \bvs_0}
    +  \norm{\Ldeuxd[T]}{ \bvs'}
    \\
    &\lesssim
    \norm{\Ldeuxd[T]}{{\bv}_T  - \bz_0 }
    +
    \norm{\Ldeuxd[T]}{\bz_0  -  \hat{\bz}_0}
    +
    \norm{\Ldeuxd[T]}{\hat{\bz} -  \bvs_0}
    +  \norm{\Ldeuxd[T]}{ \bvs'}
    \\
    &
    \lesssim
    \norm{\Ldeuxd[T]}{\bv_T }
    +
     h_T
    \|\uline{\bv}_T\|_{1,T}
    +
    \sum_{F\in \Fh[T]}h_F^{\frac{1}{2}} \norm{\Ldeuxd[F]}{\bv_F},    
  \end{align*}
  where we have used the definition of $\hat{\bz}_1$ along with the bound \eqref{eq:rtn:golyc:d} in the third step, a triangle inequality in the fourth step, and the bounds \eqref{ineq:vprime.l2}, and \eqref{eq:RT-proof:T1:bound:2}--\eqref{eq:norm:v0:zhat:final} in the last step.
  Inserting $\pm{\bv}_T$ into the left-hand side of \eqref{eq:bound.solddarcy} and using a triangle inequality followed by the above estimate, \eqref{eq:bound.solddarcy} follows.
  \smallskip\\
  \underline{(i.C) \emph{Proof  of the bound \eqref{ineq:rtn:bound}}}.
  Recalling that $\uline{\bI}_T^k$ is obtained restricting the global interpolator \eqref{eq:Ih} to an element $T$, letting $\hat{\bvs} \coloneq  {\bR}_{T}^k(\uline{\bI}_T^k\bv_T)$, and using the triangle inequality, we get that
  \begin{equation}
    \label{ineq:rtn.des}
    \norm{\Ldeuxd[T]}{\bv_T - {\bR}_{T}^k\uline{\bv}_T  }
    \leq
    \norm{\Ldeuxd[T]}{\bv_T -  \hat{\bvs} }
    +  \norm{\Ldeuxd[T]}{ \hat{\bvs} -  {\bR}_{T}^k\uline{\bv}_T  }
    \eqcolon \mathfrak{T}_1+\mathfrak{T}_2.
  \end{equation}

  By condition \eqref{eq:darcyT:weak:a},  we have that $\DIV \hat{\bvs} = D_T^k(\uline{\bI}_T^k\bv_T) \in \Poly{k}(T)\subset\Poly{k}(\Ths[T])$.
  But, since $\bv_T\in\Poly{k}(T)$, the commutation property \eqref{eq:DT.commuting} gives $D_T^k(\uline{\bI}_T^k \bv_T)  = \pi_T^k(\DIV \bv_T)=\DIV \bv_T$, so that $\DIV(\hat{\bvs}- \bv_T) = 0$.
In addition, by  \eqref{eq:darcyT:weak:ab} we have $\bpi^{{\rm c},k-1}_{\Goly{},T}(\hat{\bvs}- \bv_T)=0$,
 and then observing that $\hat{\bvs}-\bv_T\in\RTN{k}_0(\Ths[T])$, and taking $\bws = \hat{\bvs}- \bv_T$ in \eqref{eq:darcyT:weak:b}, it is inferred that $\norm{\Ldeuxd[T]}{\hat{\bvs}- \bv_T}^2=0$, hence $\mathfrak{T}_1 = 0$.
  
  Let us now estimate the term $\mathfrak{T}_2$.
  By linearity of $ {\bR}_{T}^k$, we can write $\mathfrak{T}_2 = \norm{\Ldeuxd[T]}{ {\bR}_{T}^k(\uline{\bI}_T^k\bv_T - \uline{\bv}_T )}$.
  Hence, using the bound  \eqref{eq:bound.solddarcy}, the fact that  $(\uline{\bI}_T^k\bv_T - \uline{\bv}_T)_T = 0$ and $(\uline{\bI}_T^k\bv_T - \uline{\bv}_T)_F = (\bv_T-\bv_F)$ for all $F\in\Fh[T]$, and recalling the definition \eqref{eq:norm.1T}, we can write
  \begin{align*}
    \mathfrak{T}_2
    &\lesssim  h_T \norm{1,T}{\uline{\bI}_T^k\bv_T - \uline{\bv}_T }
    + \sum_{F\in \Fh[T]}h_F^{\frac{1}{2}}  \norm{\Ldeuxd[F]}{\bv_T- \bv_F}\\
    &=
    h_T \seminorm{1,\partial T}{\uline{\bv}_T} +  \sum_{F\in \Fh[T]}h_Fh_F^{-\frac{1}{2}}  \norm{\Ldeuxd[F]}{\bv_T- \bv_F}
    {\lesssim}
    h_T \seminorm{1,\partial T}{\uline{\bv}_T},
  \end{align*}
  where we have used the inequality $h_F\leq h_T$ in the last step.
  Plugging this last bound along with  $\mathfrak{T}_1=0$ into \eqref{ineq:rtn.des}, the conclusion follows.
  \medskip\\
  \emph{(ii) Approximation.}
%%%%%%%%%%%%%%%%%%%%%%%%%%%%%%%%%%%%%%%%%%%%%%%%%%%%%%%%%%%%%%%%%%%%%%%%%%%5
  To prove the approximation estimate  \eqref{ineq:rtn:approx}, let $T\in\Th$ and denote,  for the sake of brevity, by $\hat{\uline{\bv}}_T\coloneq\uline{\bI}_T^k\bv$ the interpolate of $\bv$ on $\uline{\bU}_T^k$.
  We begin using the triangle inequality to write
  \begin{align}
    \norm{\Ldeuxd[T]}{\bv  - {\bR}_{T}^k(\hat{\uline{\bv}}_T)}
    &\leq
    \norm{\Ldeuxd[T]}{\bv  - \hat{\bv}_T  }
    +
    \norm{\Ldeuxd[T]}{\hat{\bv}_T  - {\bR}_{T}^k(\hat{\uline{\bv}}_T)  }\notag\\
    &\lesssim
    \norm{\Ldeuxd[T]}{\bv  - \bpi_T^k{\bv}  }
    +
    h_T   \seminorm{1,\partial T}{\hat{\uline{\bv}}_T}
    \eqcolon \mathfrak{T}_1 + \mathfrak{T}_2,
    \label{ineq:rtn.cons.T1.T2}
  \end{align}
  where in the last step we have used the definition of $\uline{\bI}_T^k$ and the bound \eqref{ineq:rtn:bound} for the first and second terms, respectively.
  To  bound $\mathfrak{T}_1$ we use  \eqref{eq:l2proj:error:cell} with $(l,m,r,s)=(k,0,2,k+1)$, so we get
\begin{equation}\label{ineq:rtn.cons.bound.T1}
\mathfrak{T}_1 \lesssim h_T^{k+1}   \seminorm{H^{k+1}(T)^3}{\bv}.
\end{equation}
Now to bound $\mathfrak{T}_2$ we first take the square, use the definition \eqref{eq:norm.1T} of the boundary seminorm and the equivalence $h_T\simeq h_F$ (valid for regular meshes) to obtain
\begin{align*}
(\mathfrak{T}_2)^2 
&= h_T^2\sum_{F\in \calF_T}\! h_F^{-1}\| \hat{{\bv}}_F-\hat{{\bv}}_T \|^2_{L^2(F)^3}
\lesssim
\sum_{F\in \calF_T}\! h_T \left(\| \bpi_F^k\bv- \bv \|^2_{L^2(F)^3}
+
 \|  \bv -\bpi_T^k\bv \|^2_{L^2(F)^3}\right),
\end{align*}
where in the last step we have used the Young inequality.
Now using a triangle inequality and  standard properties of the 
$L^2$-projectors $\bpi_F^k$ and $\bpi_T^k$ on $F$,  we have
\begin{equation}\label{ineq.piF.piT}
\| \bpi_F^k\bv- \bv \|^2_{L^2(F)^3}
=
\inf_{\bw \in \Poly{k}(F)^3}\| \bw - \bv \|^2_{L^2(F)^3}
\leq\|\bpi_T^k\bv -\bv \|^2_{L^2(F)^3}.
\end{equation}
Thus using this  for $F\in \calF_T$, the bound \eqref{ineq:l2.bound}, and then taking the square root, it is inferred that
$$
\mathfrak{T}_2 
\lesssim
\sum_{F\in \calF_T}\! h_T^\frac{1}{2}  \|  \bv -\bpi_T^k\bv \|_{L^2(F)^3}.
$$
Finally,  using \eqref{eq:l2proj:error:faces} with $(l,m,r,s)=(k,0,2,k+1)$ to bound $\mathfrak{T}_2$ along with \eqref{ineq:rtn.cons.bound.T1}, and plugging the result into \eqref{ineq:rtn.cons.T1.T2}, we conclude.
%%%%%%%%%%%%%%%%%%%%%%%%%%%%%%%%%%%%%%%%%%%%%%%%%%%%%%%%%%%%%%%%%%%%%%%%%%%%%%%%%%%
\medskip\\
\emph{(iii) Consistency.}
To simplify the notation let us define the space $\Goly{k-1}(T)\coloneq \GRAD \Poly{k}(T)^3$, and let 
$\bpi^{k-1}_{\Goly{},T}$ its $L^2$-orthogonal projector.
In addition, let $\bq\coloneq\bpi_T^{k-1}(\bR_{T}^k\uline{\bv}_T- \bv_T)$. As mentioned before,
the decomposition \eqref{eq:kz.decomp} is not necessarily orthogonal; nevertheless, by \cite[Lemma 1]{Di-Pietro.Droniou:21} there exists
a \emph{recovery operator} $\frakR_{\Goly{},\Goly{{\rm c}}}:\Goly{k-1}(T)\times\Goly{{\rm c},k-1}(T)\rightarrow \Poly{k-1}(T)^3$ 
such that $\bq=\frakR_{\Goly{},\Goly{{\rm c}}}(\bpi^{k-1}_{\Goly{},T}\bq,{\bpi^{{\rm c},k-1}_{\Goly,T}\bq})$, 
and $\norm{L^2(T)^3}{\bq}\simeq\norm{L^2(T)^3}{\bpi^{k-1}_{\Goly{},T}\bq}+\norm{L^2(T)^3}{{\bpi^{{\rm c},k-1}_{\Goly,T}\bq}}$. 
Using this last equation and the linearity of $\bpi_T^{k-1}$, it is enough to show that
$\norm{L^2(T)^3}{\bpi^{k-1}_{\Goly{},T}\bq}=\norm{L^2(T)^3}{{\bpi^{{\rm c},k-1}_{\Goly,T}\bq}}=0$.
Since ${\Goly{{\rm c},k-1}}\subset\Poly{k-1}(T)^3$, we have
${\bpi^{{\rm c},k-1}_{\Goly,T}}={\bpi^{{\rm c},k-1}_{\Goly,T}}\circ\bpi^{k-1}_T$,
and by
the equation {\eqref{eq:darcyT:weak:ab}}, we infer that $\norm{L^2(T)^3}{\bpi^{{\rm c},k-1}_{\Goly,T}\bq}=0$.
  Using \eqref{eq:darcyT:weak:a} along with an integration by parts and the boundary condition  \eqref{eq:darcyT:weak:bd}, we infer that, for all $\phi\in\Poly{k}(T)$,
\begin{align*}
  - \int_{T}  \bR_T^k\uline{\bv}_T \cdot \GRAD \phi
  + \cancel{\sum_{F\in \Th[T]}\int_F({\bv}_F \cdot\bn_{TF})\phi}
  = \int_{T} (D_T^k\uline{\bv}_T) \phi
  =  - \int_T \bv_T\cdot \GRAD \phi
  + \cancel{\sum_{F\in \calF_T} \int_F (\bv_F \cdot \bn_{TF}) \phi},
\end{align*}
where in the last step we have used the definition \eqref{eq:hho:div:op} of $D_T^k$.
This shows that  $\bpi^{k-1}_{\Goly{},T}(\bR_T^k\uline{\bv}_T  - {\bv}_T)=0$.
Finally, using
$\bpi^{k-1}_{\Goly{},T}=\bpi^{k-1}_{\Goly{},T}\circ\bpi^{k-1}_T$,
we obtain
$\norm{L^2(T)^3}{\bpi^{k-1}_{\Goly{},T}\bq}=0$,
and \eqref{ineq:rtn:consis} follows.
\end{proof}
\begin{remark}[Convexity assumption]\label{rem:convex.a}
The convexity assumption introduced in Section \ref{sec:setting:mesh}  can be relaxed if instead we make the following assumptions for $T\in\Th$:
\begin{enumerate}
        \item There exist a point $\widetilde{\bx}_T \in T$  such that $T$ is a star-shaped with respect to it.
        \item There exists $\tau \in \Ths[T]$ such that contains the ball  $\boldsymbol{\mathcal{B}}(\widetilde{\bx}_T,r_\tau)$  where $r_\tau$ denotes the inradius of $\tau$.
\end{enumerate}

Then instead of solving the problems \eqref{eq:neumann:weak}--\eqref{eq:neumann:strong},  
invoke the Lemma III.3.1 of \cite{Galdi:2011} 
to obtain $z_0\in H_0^1(T)^3$  such that $\DIV \bz_0=D_T^k\uline{\bv}_T  - \DIV \bvs'$ with $\norm{H^1(T)^3}{\bz_0}\lesssim\norm{L^2(T)}{D_T^k\uline{\bv}_T  - \DIV \bvs'}$,
and use  \cite[Eq. (II.5.5)]{Galdi:2011} to get a Poincaré-like inequality $\norm{\Ldeuxd[T]}{\bz_0}\lesssim h_T\seminorm{H^1(T)^3}{\bz_0}$
and use it in \eqref{eq:RT-proof:T1:bound} 
in the proof  of item (i) of Lemma \ref{lemm:rtn}.
\end{remark}

Let $\RTN{k}(\Ths)$ denote the global ($\Hdiv$-conforming) Raviart--Thomas--N\'ed\'elec space on $\Ths$.
We define the \emph{global velocity reconstruction} $\bR_h^k: \uline{\bU}_h^k \rightarrow \RTN{k}(\Ths)$ patching the local contributions:
For all  $\uline{\bv}_h \in \uline{\bU}_h^k$,
\begin{align*}
  (\bR_h^k\uline{\bv}_h)_{|T}
  \coloneq 
  \bR_T^k\uline{\bv}_T \qquad\forall T \in {\cal{T}}_h.
\end{align*}
Note that $\bR_h^k\uline{\bv}_h$ is well-defined, since its normal components across each mesh interface are continuous as a consequence of  \eqref{eq:darcyT:weak:bd}  combined with the single-valuedness of interface unknowns.

\begin{proposition}[Sobolev inequalities for the velocity reconstruction]\label{prop:rtn:Lp.bound}
  It holds, for all $r\in[1,6]$ and all $\uline{\bv}_h\in\uline{\bU}_{h,0}^k$,
  \begin{equation}
    \| \bR_h^k\uline{\bv}_h\|_{L^r(\Omega)^3} \lesssim  \| \uline{\bv}_h \|_{1,h},
    \label{eq:rtn:Lp.bound}
  \end{equation}
  where the hidden constant is independent of both $h$ and $\uline{\bv}_h$, but possibly depends on $\Omega$, $k$, $r$, and the mesh regularity parameter.
\end{proposition}
\begin{proof}
    Let a mesh element $T\in\calT_h$ be fixed.
    Inserting $\pm\bv_T$ into the norm and using a triangle inequality, we can write
    \begin{equation}
    \begin{split}  
    \| \bR_T^k\uline{\bv}_T\|_{L^r(T)^3}
    &\leq \| \bR_T^k\uline{\bv}_T - {\bv}_T \|_{L^r(T)^3}  + \| {\bv}_T \|_{L^r(T)^3}
    \\
    &= 
    \left( 
     \sum _{\tau\in\Ths[T]}
    \| \bR_T^k\uline{\bv}_T - {\bv}_T \|_{L^r(\tau)^3}^r
\right)^\frac{1}{r}
    + \| {\bv}_T \|_{L^r(T)^3}.
    \label{eq:rtn:L4:split}
  \end{split}  
  \end{equation}
    From the discrete Lebesgue embeddings proved in \cite[Lemma 1.25]{Di-Pietro.Droniou:20}, it follows that, for all $(\alpha,\beta)\in [1,+\infty]$, all $X\in\Th\cup\Ths$, and all $\zeta\in\Poly{l}(X)$ for $l\geq0$,
  \begin{equation}
    \|\zeta\|_{L^{\alpha}(X)}\lesssim h_X^{\frac{3}{\alpha}-\frac{3}{\beta}}  \|\zeta\|_{L^{\beta}(X)},
    \label{eq:rev:emb:T}
  \end{equation}
  with hidden constant independent of $h$, $X$, and $\zeta$, but possibly depending on $l$, $\alpha$, $\beta$, and the mesh regularity parameter.
  Since  $(\bR_T^k\uline{\bv}_T - {\bv}_T)_{|\tau}\in \Poly{k+1}(\tau)^3$, we use
  \eqref{eq:rev:emb:T} for $(X,l,\alpha,\beta)=(\tau,k+1,r,2)$ in the term in parentheses of  \eqref{eq:rtn:L4:split} to write
  \begin{equation}
    \begin{aligned}\label{eq:rtn:l4:diff:bound}
      \sum _{\tau\in\Ths[T]}
      \| \bR_T^k\uline{\bv}_T - {\bv}_T \|_{L^2(\tau)^r} ^r
      &\lesssim 
      \sum _{\tau\in\Ths[T]}
      h_\tau^{r\left(\frac{3}{r}-\frac{3}{2}\right)} \| \bR_T^k\uline{\bv}_T - {\bv}_T \|_{L^2(\tau)^3}^r
      \\
      &\lesssim 
      h_T^{r\left(\frac{3}{r}-\frac{3}{2}\right)} \| \bR_T^k\uline{\bv}_T - {\bv}_T \|_{L^2(T)^3}^r
      \lesssim 
      h_T^{r\left(\frac{3}{r}-\frac{1}{2}\right)} \seminorm{1,\partial T}{\uline{\bv}_T}^r,
    \end{aligned}
  \end{equation}
  where  we have used $\tau\subset T$ and $h_\tau\le h_T$ for all $\tau\in\Ths[T]$ along with the uniform bound \eqref{ineq:card.IT.F} on $\card(\Ths[T])$ in the second step, and the estimate \eqref{ineq:rtn:bound} to conclude.
  Plugging \eqref{eq:rtn:l4:diff:bound} into \eqref{eq:rtn:L4:split}, raising the resulting inequality to the $r$-th power, using the inequality $(a+b)^r\lesssim a^r+b^r$ valid for any nonnegative real numbers $a$ and $b$, and summing over $T\in\calT_h$, we get
  \begin{equation*}%% \label{eq:rtn:Lp.bound:basic}
    \| \bR_T^k\uline{\bv}_T\|_{L^r(\Omega)^3}^r
    \lesssim \sum_{T\in\calT_h} h_T^{\frac{6-r}2}| \uline{\bv}_T |_{1,\partial T}^r
    +  \| {\bv}_h \|_{L^r(\Omega)^3}^r.
  \end{equation*}
  The proof now continues as that of \cite[Proposition 3]{Castanon-Quiroz.Di-Pietro:20}.
  The details are omitted for the sake of conciseness.
\end{proof}

%------------------------------------------------------------------------------%

\subsection{Gradient reconstruction on a submesh}\label{sec:discrete.setting:gradient.reconstruction}

Let an element $T\in\calT_h$ be fixed.
For every face $\sigma\in\Fhs[T]^{\rm i}$, we introduce an arbitrary but fixed ordering of the elements $\tau_1$ and $\tau_2$ such that $\sigma \subset \partial \tau_1 \cap \partial \tau_2$, and let $\bn_\sigma\coloneq\bn_{\tau_1\sigma}=-\bn_{\tau_2\sigma}$, where  $\bn_{\tau_i\sigma}, i\in\{1,2\}$,
denotes the unit vector normal to $\sigma$ pointing out of $\tau_i$ (see Figure \ref{fig:simplices.faces.T.b}).
With this convention, for every scalar-valued function $\zeta$ admitting a possibly two-valued trace on $\sigma$, we define the jump of $\zeta$ across $\sigma$ as
\begin{equation}
  \label{def:jump.trace.op}
  \llbracket \zeta\rrbracket_\sigma \coloneq \zeta_{|\tau_1} - \zeta_{|\tau_2}.
\end{equation}
When applied to vector- or tensor-valued functions, the jump  operator acts component-wise.

For any polyomial degree $l\geq0$, we then define the local gradient reconstruction $\bG_{\Ths[T]}^{l}:\uline{\bU}_T^k \rightarrow \Poly{l}(\Ths[T])^{3\times 3}$ such that, for all $\uline{\bv}_T \in \uline{\bU}_T^k$ and all $\btau \in \Poly{l}(\Ths[T])^{3\times 3}$,
\begin{subequations}
  \label{eq:disc:grad}
  \begin{alignat}{2}
    \int_T \bG^l_{\Ths[T]} \uline{\bv}_T: \btau
    &=
    \int_T \GRAD \bv_T :  \btau + \sum_{F\in\calF_T}\int_F (\bv_F-\bv_T)\cdot\btau\bn_{TF}
    \label{eq:disc:grad:b}
    \\
    &=
    - \int_T \bv_T\cdot (\DIV \btau)
    +\sum_{\sigma\in\Fhs[T]^{\rm i}}\int_\sigma \bv_T\cdot  \llbracket \btau\rrbracket_\sigma \bn_\sigma
    + \sum_{F\in\calF_T}\int_F \bv_F\cdot\btau\bn_{TF},
    \label{eq:disc:grad:a}
  \end{alignat}
\end{subequations}
where we have used an integration by parts to pass to the second line.
The above definition is an extension of the operator $\bG_{T}^{l}:\uline{\bU}_T^k \rightarrow \Poly{l}(T)^{3\times 3}$  introduced in \cite{Castanon-Quiroz.Di-Pietro:20,Di-Pietro.Krell:18}, which is defined using  $\Poly{l}(T)^{3\times 3}$ instead of $\Poly{l}(\Ths[T])^{3\times 3}$ as a test space, and thus we  have
$\boldsymbol{\pi}_T^l\bG_{\Ths[T]}^{l} = \bG_{T}^{l}$.
The gradient reconstruction $\bG_{T}^{k}$ will be used in the viscous term, while the enriched gradient reconstruction $\bG_{\Ths[T]}^{2(k+1)}$
will be used  in the convective term (see Section \ref{sec:discrete.problem:convective.term}).

\begin{lemma}[Properties of $\bG_{\Ths[T]}^{l}$]
  \label{lemm:disc:grad}
  The  operator $\bG_{\Ths[T]}^{l}$ has the following properties:
  \begin{enumerate}
  \item \emph{Boundedness.} For all $\uline{\bv}_T \in \uline{\bU}_T^k$,
    it holds  
    \begin{equation}
      \|\bG^l_{\Ths[T]}\uline{\bv}_T \|_{L^2(T)^{3\times 3}}
      \lesssim
      \| \uline{\bv}_T  \|_{1,T}.
      \label{eq:disc:grad:bound}
    \end{equation}
  \item \emph{Consistency.}
    For all $\bv\in H^{k+1}(T)^3$ and all $l>k$, it holds,
    \begin{equation}
      \| \bG^l_{\Ths[T]}\uline{\bI}_T^k{\bv}  - \nabla \bv \|_{L^2(T)^{3\times 3}}
      \lesssim
      h_T^{k}|\bv|_{H^{k+1}(T)^3}.
      \label{eq:disc:grad:consistency}
    \end{equation}
  \end{enumerate}
\end{lemma}

\begin{proof}
  \emph{(i) Boundedness.}
  The proof is the same as that of \cite[Proposition 1]{Di-Pietro.Krell:18}.
  \medskip\\
  \emph{(ii) Consistency.}
  Let $\bv \in H^{k+1}(T)^3$. For all $T \in \Th$, using $\uline{\bI}_T^k\bv = (\bpi_T^k\bv,(\bpi_F^k\bv_{|F})_{F\in\Fh[T]})$ into \eqref{eq:disc:grad:a} for the first term and an integration by parts for the second term, we obtain, for all $\btau\in \Poly{k}(\Ths[T])^{3\times3}$,
  \begin{equation}    
    \begin{aligned}
      \label{eq:disc:grad:proof}
      \int_T (\bG_{\Ths[T]}^l \uline{\bI}_T^k\bv-\GRAD \bv): \btau
      =&
      - \int_T (\bpi_T^k\bv- \bv)\cdot (\DIV \btau)
      + \sum_{F\in\calF_T}\int_F (\bpi_F^k\bv- \bv)\cdot\btau\bn_{TF}\\
      &
      +\sum_{\sigma\in\Fhs[T]^{\rm i}}\int_\sigma (\bpi_T^k\bv- \bv)\cdot  \llbracket \btau\rrbracket_\sigma \bn_\sigma
      \eqcolon\mathfrak{T}_1 + \mathfrak{T}_2 + \mathfrak{T}_3.
    \end{aligned}
  \end{equation}
  We now proceed to bound the terms in the right-hand side.
  \smallskip
  Using Cauchy--Schwarz and discrete inverse inequalities along with the approximation properties \eqref{eq:l2proj:error:cell} of  $\bpi_T^k$ with ${(l,r,m,s)} = {(k,2,0,k+1)}$ we obtain, for the first term,
  \begin{equation}\label{eq:disc:grad:T1}
    |\mathfrak{T}_1|\lesssim h_T^k\seminorm{H^{k+1}(T)^3}{\bv}\norm{L^2(T)^{3\times 3}}{\btau}.
  \end{equation}
  \smallskip
  For the second term, we use a H\"older inequality with exponents $(2,2,\infty)$ along with $\norm{L^\infty(F)^3}{\bn_{TF}}=1$ to write
  \begin{multline}\label{eq:disc:grad:T2}
    |\mathfrak{T}_2|
    \le\sum_{F\in\Fh[T]}\norm{L^2(F)^3}{\bpi_F^k\bv- \bv}\norm{L^2(F)^{3\times 3}}{\btau}
    \\
    \lesssim h_T^{-\nicefrac12}\left(
    \sum_{F\in\Fh[T]}\norm{L^2(F)^3}{\bpi_T^k\bv- \bv}^2
    \right)^{\nicefrac12}\norm{L^2(T)^{3\times 3}}{\btau}
    \lesssim
    h_T^k\seminorm{H^{k+1}(T)^3}{\bv}\norm{L^2(T)^{3\times 3}}{\btau},
  \end{multline}
  where we have used the inequality \eqref{ineq.piF.piT} together with a discrete trace inequality in the second step and the trace approximation properties \eqref{eq:l2proj:error:faces} of $\bpi_T^k$ with $(l,r,m,s) = (k,2,0,k+1)$ to conclude.\\
 \smallskip
  Let us now consider the third term {in \eqref{eq:disc:grad:proof}}.
  Recalling the definition \eqref{def:jump.trace.op} of the jump operator, we bound each integral over $\sigma\in\Fhs[T]^{\rm i}$ as follows:
  \begin{align*}
    \left | \int_\sigma (\bpi_T^k\bv- \bv)\cdot  \llbracket \btau\rrbracket_\sigma \bn_\sigma  \right |
    &\le
    \sum _{i=1}^2 \left |      \int_\sigma (\bpi_T^k\bv- \bv)\cdot\left( \btau_{|\tau_i} \bn_{\tau_i\sigma} \right) \right |\\
    &\leq
    \sum _{i=1}^2   \norm{L^2(\sigma)^3}{\bpi_T^k\bv- \bv} \norm{L^2(\sigma)^{3\times3}}{\btau_{|\tau_i}}\\
    & 
    \lesssim 
    \sum _{i=1}^2  
    \left(  h_{\sigma}^{-1}\norm{L^2(\tau_i)^3}{\bpi_T^k\bv- \bv}   + \seminorm{H^1(\tau_i)^3}{\bpi_T^k\bv- \bv} \right)
    \norm{L^2(\tau_i)^{3\times3}}{\btau}\\
    & 
    \lesssim \left( h_T^{-1}\norm{L^2(T)^3}{\bpi_T^k\bv- \bv}  + \seminorm{H^1(T)^3}{\bpi_T^k\bv- \bv} \right) \norm{L^2(T)^{3\times3}}{\btau}\\%     
    &
    \lesssim  h_T^k|\bv|_{H^{k+1}(T)^3} \norm{L^2(T)^{3\times3}}{\btau},
  \end{align*}
  where we have started with a triangle inequality,
  used Cauchy--Schwarz and H\"older inequalities (the latter with exponents $(2,\infty)$) along with $\|\bn_{\tau_i\sigma}\|_{L^\infty(\sigma)^3}=1$ in the  second step,
  local continuous and discrete trace inequalities on the submesh for the first and second factor, respectively, in the third step,
  and the fact that $\tau_i\subset T$ for $i\in\{1,2\}$ along with the first geometric bound in \eqref{ineq:card.IT.F} and $h_\sigma^{-1}\lesssim h_T^{-1}$ (consequence of mesh regularity) in the fourth step.
  The conclusion follows using the approximation properties \eqref{eq:l2proj:error:cell} of  $\bpi_T^k$ with $(l,r,m,s)=(k,2,0,k+1)$ for the first term in parenthesis and $(l,r,m,s)=(k,2,1,k+1)$ for the second one.
  Gathering the above estimates and observing that $\card(\Fhs[T]^{\rm i})\le 4\card(\Ths[T])\lesssim {1}$ by \eqref{ineq:card.IT.F}, we obtain
  \begin{equation}\label{eq:disc:grad:T3}
    |\mathfrak{T}_3|{\lesssim} h_T^k|\bv|_{H^{k+1}(T)^3} \norm{L^2(T)^{3\times3}}{\btau}.
  \end{equation}
  \smallskip

  Setting $\btau=\bG_{\Ths[T]}^l \uline{\bI}_T^k\bv-\bpi ^l_{\Ths[T]}\GRAD \bv$ in \eqref{eq:disc:grad:proof}, using the bounds \eqref{eq:disc:grad:T1}--\eqref{eq:disc:grad:T3}, and simplifying yields
  $$
  \| \bG_{\Ths[T]}^l\uline{\bI}_T^k{\bv}  - \bpi ^l_{\Ths[T]} \nabla \bv \|_{L^2(T)^{3\times 3}}
  \lesssim 
  h_T^k|\bv|_{H^{k+1}(T)^3}.
  $$
  Therefore, using a triangle inequality and the approximation properties \eqref{eq:l2proj:error:cell}, valid as well for  $\bpi_\tau^l$, with $(l,r,m,s)=(k,2,0,k+1)$ along with $h_\tau\leq h_T$, we infer
  \[
    \| \bG_{\Ths[T]}^l\uline{\bI}_T^k{\bv}  - \nabla \bv \|_{L^2(T)^{3\times 3}}
    \leq
    \| \bG_{\Ths[T]}^l\uline{\bI}_T^k{\bv}  - \bpi ^l_{\Ths[T]} \nabla \bv \|_{L^2(T)^{3\times 3}}
    +    
    \| \bpi ^l_{\Ths[T]} \nabla \bv - \nabla \bv \|_{L^2(T)^{3\times 3}}\\
    \lesssim   h_T^k|\bv|_{H^{k+1}(T)^3}.\qedhere
    \]
\end{proof}

%------------------------------------------------------------------------------%
%------------------------------------------------------------------------------%

\section{Discrete problem}\label{sec:discrete.problem}

\subsection{Viscous term and pressure-velocity coupling}

The viscous term and the pressure-velocity coupling are the same as in the standard HHO method; see, e.g., \cite{Di-Pietro.Krell:18,Botti.Di-Pietro.ea:19}.
We briefly recall them here to make the exposition self-contained.
\smallskip

The viscous bilinear form $a_h$: $\uline{\bU}_h^k \times \uline{\bU}_h^k\rightarrow \mathbb{R}$ is such that, for all $\uline{\bw}_h, \uline{\bv}_h, \in \uline{\bU}_h^k$,
\begin{equation*}%% \label{eq:ah}
  a_h(\uline{\bw}_h,\uline{\bv}_h)\coloneq
  \sum _{T \in \calT_h}\left(
  \int_T \bG_T^k \uline{\bw}_T  : \bG_T^k \uline{\bv}_T
  + s_T(\uline{\bw}_T,\uline{\bv}_T)
  \right),
\end{equation*}
where, for any $T\in\calT_h$, $s_T: \uline{\bU}_T^k \times \uline{\bU}_T^k\rightarrow \mathbb{R}$ 
denotes a local stabilization bilinear form designed according to the principles of \cite[Assumption 2.4]{Di-Pietro.Droniou:20},
so that, in particular, there exists $C_a>0$ independent of $h$ 
(and, clearly, also of $\nu$ and $\lambda$) such that, for all $\uline{\bv}_h \in \uline{\bU}_{h}^k$,
\begin{equation}
  C_a\| \uline{\bv}_h  \|_{1,h}^2
  \le a_h(\uline{\bv}_h,\uline{\bv}_h)
  \le C_a^{-1} \| \uline{\bv}_h  \|_{1,h}^2.
  \label{eq:ns:stab.h}
\end{equation}
\smallskip

Recalling the definition  \eqref{eq:hho:div:op}  of the local divergence $D_T^k$, the global pressure-velocity coupling bilinear form $b_h:  \uline{\bU}_{h,0}^k \times \Poly{k}(\calT_h)\rightarrow \mathbb{R}$ is such that, for all $(\uline{\bv}_h,q_h) \in  \uline{\bU}_{h,0}^k \times \Poly{k}(\calT_h)$,
$$
b_h(\uline{\bv}_h,q_h):= - \sum _{T\in \calT_h} \int_T D_T^k \uline{\bv}_T\, q_T,
$$
where $q_T \coloneq q_{h|T}$. 
The properties of $b_h$ relevant for the analysis can be found in \cite[Lemma 8.12]{Di-Pietro.Droniou:20}.

\subsection{Body force}\label{subsec:div:rtn}

The discretization of the body force leverages the new divergence-preserving velocity reconstruction introduced in Section \ref{sec:discrete.setting:velocity.reconstruction}.
Specifically, we introduce the bilinear form $\ell_h: \Ldeuxd \times \uline{\bU}_h^k \rightarrow  \mathbb{R}$ such that, for any $\boldsymbol{\phi} \in \Ldeuxd$ and any $\uline{\bv}_h \in \uline{\bU}_h^k$,
\begin{equation*}%%   \label{eq:lh:form}
  \ell_h(\boldsymbol{\phi},\uline{\bv}_h) \coloneq \int_\Omega \boldsymbol{\phi}\cdot \bR_h^k\underline{\bv}_h.
\end{equation*}

\begin{lemma}[Properties of $\ell_h$]\label{lem:ell.h}
  \label{lem:ns:lh}
  The bilinear form $\ell_h$ has the following properties:
  \begin{enumerate}[(i)]
  \item \emph{Velocity invariance.} Recalling the Hodge decomposition \eqref{eq:hodge.f} of $\bef$, it holds
    \begin{equation}
      \ell_h(\bg+\lambda\nabla \psi,\uline{\bv}_h)
      =
      \ell_h(\bg,\uline{\bv}_h) + b_h(\uline{\bv}_h,\lambda\pi_h^k\psi)
      \qquad \forall \uline{\bv}_h \in \uline{\bU}_{h,0}^k.
      \label{eq:ns:lh:invariance}
    \end{equation}
  \item \emph{Consistency.} For all $\boldsymbol{\phi} \in \Ldeuxd\cap H^{k}(\calT_h)^3$,
   \begin{equation}
      \| {\cal{E}}_{\ell,h}(\boldsymbol{\phi};\cdot) \|_{1,h,*}
      \lesssim
      h^{k+1}|\boldsymbol{\phi}|_{H^{k}(\calT_h)^3},
      \label{eq:ns:lh:consistency}
    \end{equation}
    where the linear form ${\cal{E}}_{\ell,h}(\boldsymbol{\phi};\cdot):\uline{\bU}_{h}^k\rightarrow\mathbb{R}$, representing the body force consistency error, is such that        
    \begin{align}
      {\cal{E}}_{\ell,h}(\boldsymbol{\phi};\uline{\bv}_h)
      \coloneq 
      \ell_h( \boldsymbol{\phi},\uline{\bv}_h)
      - \int_\Omega \boldsymbol{\phi} \cdot \bv_h
      = \sum_{T\in \Th} \int_T \bphi \cdot  (\bR_T^k\uline{\bv}_T  - {\bv}_T).
      \label{Eh.lh}
    \end{align}
  \end{enumerate}
\end{lemma}

\begin{proof}
  \emph{(i) Velocity invariance.} The proof is the same as in  \cite[Section 4.3]{Castanon-Quiroz.Di-Pietro:20}, using the fact that $(\DIV   \bR_h^k\underline{\bv}_h)_{|T}=D_T^k\uline{\bv}_T\in\Poly{k}(T)$, which is enforced by \eqref{eq:darcyT:weak:a}.
  \medskip\\
  \emph{(ii) Consistency.} We prove the cases $k=0$ and {$k\geq 1$} separately.
  \medskip\\
  \underline{(ii.A) \emph{The case $k=0$.}}
  Taking absolute values in \eqref{Eh.lh} and using Cauchy--Schwarz inequalities along with \eqref{ineq:rtn:bound} and $h_T\le h$ for all $T\in\Th$, we can write
  $\left |  {\cal{E}}_{\ell,h}(\boldsymbol{\phi};\uline{\bv}_h) \right | \leq h  \norm{L^2(\Omega)^3}{\bphi}  \norm{1,h}{\uline{\bv}_T}$.
  Passing to the supremum over $\uline{\bv}_h \in \uline{\bU}_{h}^k$ such that $\norm{1,h}{ \uline{\bv}_h}=1$, we obtain   \eqref{eq:ns:lh:consistency}.
  \medskip\\
  \underline{(ii.B) \emph{The case $k\geq 1$.}}
  Using \eqref{ineq:rtn:consis} in \eqref{Eh.lh} and continuing with Cauchy--Schwarz inequalities, we obtain
  \begin{align*}
    \left|{\cal{E}}_{\ell,h}(\bphi;\uline{\bv}_h)\right|
    = \left|\sum_{T\in \Th} \int_T (\bphi -\bpi^{k-1} \bphi) \cdot  (\bR_T^k\uline{\bv}_T  - {\bv}_T)\right|
    \le\sum_{T\in\Th}\norm{L^2(T)^3}{\bphi -\bpi^{k-1} \bphi}\norm{L^2(T)^3}{\bR_T^k\uline{\bv}_T  - {\bv}_T }.
  \end{align*}
  Using then the approximation properties  \eqref{eq:l2proj:error:cell} of the $L^2$-projector with $(l,m,r,s)=(k-1,0,2,k)$ for the first factor and the bound  \eqref{ineq:rtn:bound} for the second, applying discrete Cauchy--Schwarz inequalities to the sums, and passing to the supremum over $\uline{\bv}_h \in \uline{\bU}_{h}^k$ such that $\norm{1,h}{ \uline{\bv}_h}=1$, \eqref{eq:ns:lh:consistency} follows.
\end{proof}
  
\subsection{Convective term}\label{sec:discrete.problem:convective.term}

To discretize the convective term, we introduce the global trilinear form $t_h:\left[\uline{\bU}_h^k\right]^3\to\mathbb{R}$ such that
\begin{subequations}\label{eq:th}
  \begin{equation}
    t_h(\uline{\bw}_h,\uline{\bv}_h,\uline{\bz}_h)
    \coloneq 
    \sum _{T \in \calT_h}   t_T(\uline{\bw}_T,\uline{\bv}_T,\uline{\bz}_T),
    \label{eq:th:glob:form}
  \end{equation}
  where, for any $T\in\calT_h$, $t_T:\left[\uline{\bU}_T^k\right]^3\rightarrow \mathbb{R}$
  is defined  as 
  \begin{equation}\label{eq:th:lcl:form}
    t_T(\uline{\bw}_T,\uline{\bv}_T,\uline{\bz}_T)
    \coloneq
    \int_T \bG_{\Ths[T]}^{2(k+1)}\uline{\bw}_T \bR_T^k\uline{\bv}_T \cdot \bR_T^k\uline{\bz}_T - \int_T  \bG_{\Ths[T]}^{2(k+1)}\uline{\bw}_T \bR_T^k\uline{\bz}_T \cdot \bR_T^k\uline{\bv}_T.
  \end{equation}
\end{subequations}

\begin{remark}[Reformulation of  $t_h$]\label{rem:th:reformulation}
  In  practice, it is not necessary to compute the piecewise gradient reconstruction operators $\bG_{\Ths[T]}^{2(k+1)}$ to evaluate $t_T$ and $t_h$.
As a matter of fact, expanding the piecewise gradient operator in \eqref{eq:th} according to its definition \eqref{eq:disc:grad:b}, we have that
  \begin{align*}  
    t_h(\uline{\bw}_h,\uline{\bv}_h,\uline{\bz}_h)=
    &\sum _{T \in \calT_h} \left [
      \int_T \nabla {\bw}_T \bR_T^k\uline{\bv}_T \cdot \bR_T^k\uline{\bz}_T
      - \int_T  \nabla {\bw}_T \bR_T^k\uline{\bz}_T \cdot \bR_T^k\uline{\bv}_T
      \right]\\
    &+\sum _{T \in \calT_h} \sum _{F \in \calF_T}
    \int_F (\bw_F- \bw_T) \cdot \bR_T^k\uline{\bz}_T(\bR_T^k\uline{\bv}_T   \cdot \bn_{TF})
    \\
    &
    -\sum _{T \in \calT_h} \sum _{F \in \calF_T}
    \int_F (\bw_F- \bw_T) \cdot \bR_T^k\uline{\bv}_T(\bR_T^k\uline{\bz}_T   \cdot \bn_{TF}).
  \end{align*}  
\end{remark}
The properties of $t_h$ relevant for the analysis are contained in the following lemma.
\begin{lemma}[Properties of $t_h$]
  \label{lem:ns:th}
  The trilinear form $t_h$ has the following properties:
  \begin{enumerate}
  \item \emph{Non-dissipativity.} For all $\uline{\bw}_h,\uline{\bv}_h \in \uline{\bU}_h^k$,
    it holds that    
    \begin{equation}  
      t_h(\uline{\bw}_h,\uline{\bv}_h,\uline{\bv}_h)=0.
      \label{ns:th:skewsym}
    \end{equation}
  \item \emph{Boundedness.} There exists a real number $C_t>0$ independent of $h$ (and, clearly, also of $\nu$ and $\lambda$) such that, for all $\uline{\bw}_h,\uline{\bv}_h,\uline{\bz}_h \in \uline{\bU}_h^k$,
    \begin{equation}
      |t_h(\uline{\bw}_h,\uline{\bv}_h,\uline{\bz}_h)|
      \le C_t
      \| \uline{\bw}_h \|_{1,h}  \|\uline{\bv}_h\|_{1,h}\| \uline{\bz}_h\|_{1,h}.
      \label{ns:th:boundedness}
    \end{equation}
  \item \emph{Consistency.} It holds, for all $\bw \in \bU \cap W^{k+1,4}(\calT_h)^3$ such that $\DIV \bw = 0$ a.e. in $\Omega$,
    \begin{equation}
      \| {\cal{E}}_{t,h}(\bw;\cdot) \|_{1,h,*}
      \lesssim
      h^{k+1} | \bw |_{W^{k+1,4}(\calT_h)^3} \|\bw\|_{W^{1,4}(\Omega)^3},
      \label{ns:th:consistency}
    \end{equation}
    where the linear form ${\cal{E}}_{t,h}(\bw;\cdot):\uline{\bU}_h^k\rightarrow\mathbb{R}$ representing the consistency error is such that, for all $\uline{\bz}_h \in \uline{\bU}_h^k$,
    \[
      {\cal{E}}_{t,h}(\bw;\uline{\bz}_h)
      \coloneq
      \ell_h(  (\ROT \bw) \times \bw ,\uline{\bz}_h)
      -
      t_h(\uline{\bI}_h^k\bw,\uline{\bI}_h^k\bw,\uline{\bz}_h).
    \]
  \end{enumerate}
\end{lemma}
\begin{proof}
  (i) \emph{Non-dissipativity.} Immediate consequence of the definition \eqref{eq:th} of $t_h$.
  \medskip\\
  (ii) \emph{Boundedness.} The proof is similar to that of  \cite[Lemma 7.ii]{Castanon-Quiroz.Di-Pietro:20}  using the H\"older inequalities with exponent (2,4,4), the bound $\eqref{eq:disc:grad:bound}$, and the discrete Sobolev embedding \eqref{eq:rtn:Lp.bound} with $r=4$.
  The details are omitted for the sake of conciseness.
  \medskip\\
  (iii) \emph{Consistency.}  Let  $\uline{\hat{\bw}}_h\coloneq \uline{\bI}_h^k\bw$.
        Proceeding as in  \cite[Lemma 7.iii]{Castanon-Quiroz.Di-Pietro:20}, we obtain the following decomposition:
  \begin{equation}\label{eq:curl:all:T}
    \begin{aligned}
      {\cal{E}}_{t,h}(\bw;\uline{\bz}_h)
      &=
      \underbrace{\sum _{T\in \calT_h}\int _{T}(\bG_{\Ths[T]}^{2(k+1)}\uline{\hat{\bw}}_T -\nabla \bw)  \bR_T^k\uline{\bz}_T \cdot \bw}_{\mathfrak{T}_1}
      + \underbrace{\sum _{T\in \calT_h}\int _{T}(\nabla \bw - \bG_{\Ths[T]}^{2(k+1)}\uline{\hat{\bw}}_T  ) \bw \cdot \bR_T^k\uline{\bz}_T}_{\mathfrak{T}_2}
      \\
      &\quad
      + \underbrace{\sum _{T\in \calT_h}\int _{T}\bG_{\Ths[T]}^{2(k+1)}\uline{\hat{\bw}}_T  (\bw- \bR_T^k\uline{\hat{\bw}}_T) \cdot \bR_T^k\uline{\bz}_T}_{\mathfrak{T}_3}
      + \underbrace{\sum _{T\in \calT_h}\int _{T} \bG_{\Ths[T]}^{2(k+1)}\uline{\hat{\bw}}_T \bR_T^k\uline{\bz}_T \cdot ( \bR_T^k\uline{\hat{\bw}}_T - \bw).}_{\mathfrak{T}_4}
    \end{aligned}
  \end{equation}
  We next proceed to estimate the terms $\mathfrak{T}_1, \cdots, \mathfrak{T}_4$.
  \medskip\\
  \uline{(iii.A) \emph{Estimate of $\mathfrak{T}_1$.}}
  Following  similar steps as in \cite[Lemma 7.iii.A]{Castanon-Quiroz.Di-Pietro:20} using the  approximation properties
  \eqref{eq:disc:grad:consistency} of  $\bG_{\Ths[T]}^{2(k+1)}$  and its  definition \eqref{eq:disc:grad}, we get that
  \begin{equation}
    |\mathfrak{T}_1| 
    {\lesssim}
    h^{k+1}|\bw |_{H^{k+1}({\cal{T}}_h)^3} \| \bw  \|_{W^{1,4}(\Omega)^3} \| \uline{\bz}_h \|_{1,h}.
    \label{eq:curl:T1:bound}                 
  \end{equation}
  \medskip
  \uline{(iii.B) \emph{Estimate of $\mathfrak{T}_2$.}}
  For the term $\mathfrak{T}_2$ in  \eqref{eq:curl:all:T}, inserting $\pm\bpi_T^0\bw$ into the second factor, we get
  \begin{equation}
    \begin{aligned}
      \mathfrak{T}_2
      &=
      \sum _{T\in \calT_h} \int _{T}
      (\nabla \bw - \bG_{\Ths[T]}^{2(k+1)}\uline{\hat{\bw}}_T ) (\bw -\bpi_T^0\bw) \cdot  \bR_T^k\uline{\bz}_T
      + \sum _{T\in \calT_h} \int _{T}
      (\nabla \bw - \bG_{\Ths[T]}^{2(k+1)}\uline{\hat{\bw}}_T )  \bpi_T^0\bw  \cdot \bR_T^k\uline{\bz}_T
      \\
      &\eqcolon \mathfrak{T}_{2,1} + \mathfrak{T}_{2,2}.
    \end{aligned}
    \label{eq:curl:T2}                 
  \end{equation}  
  We bound  $\mathfrak{T}_{2,1}$ using  H\"older inequalities with exponents $(2,4,4)$, then
  the approximation properties \eqref{eq:disc:grad:consistency} of $\bG_{\Ths[T]}^{2(k+1)}$ 
  and \eqref{eq:l2proj:error:cell} of $\bpi_T^0$
  with $(l,m,r,s)=(0,0,4,1)$,
  and the bound \eqref{eq:rtn:Lp.bound} with $r=4$:
  \begin{equation}
    |\mathfrak{T}_{2,1}| 
    \lesssim h^{k+1}|\bw |_{H^{k+1}(\calT_h)^3}   | \bw  |_{W^{1,4}(\Omega)^3} \| \uline{\bz}_T \|_{1,h}.
    \label{eq:curl:T21}
  \end{equation}
  To estimate $\mathfrak{T}_{2,2}$ in  \eqref{eq:curl:T2}, we integrate by parts the term involving $\nabla\bw$ and we use, for each element $T \in \calT_h$, the definition  \eqref{eq:disc:grad:a} of $\bG_{\Ths[T]}^{2(k+1)}$ with $(\uline{\bv}_T, \btau) = (\uline{\hat{\bw}}_T, \bR_T^k\uline{\bz}_T \otimes \bpi_T^0\bw)$ (notice that $ \bR_T^k\uline{\bz}_T \otimes \bpi_T^0\bw \in \Poly{k+1}(\Ths[T])^{3\times 3}\subset \Poly{2(k+1)}(\Ths[T])^{3\times 3}$) to write
  \begin{equation}\label{eq:curl:T22}
    \begin{aligned}
      \mathfrak{T}_{2,2}
      &=
      -\sum _{T\in \calT_h}
      \sum _{\tau \in \Ths[T]}
      \int _{\tau} (\bw - {\bpi}_T^k\bw ) \cdot \DIV (\bR_T^k\uline{\bz}_T \otimes \bpi_T^0\bw)
      \\
      &\quad
      +
      \sum _{T \in \Th}
      \sum _{\tau \in \Ths[T]}
      \sum_{\sigma\in\Fhs[T]^{\rm i}}
      \int_\sigma (\bw - {\bpi}_T^k\bw ) \cdot  \llbracket \bR_T^k\uline{\bz}_T \otimes \bpi_T^0\bw\rrbracket_\sigma \bn_\sigma     
      \\
      &\quad
      +\sum _{T\in \calT_h}\sum _{F\in \calF_T}
      \int _{F} ( \bw - {\bpi}_F^k\bw) \cdot  (\bR_T^k\uline{\bz}_T \otimes  \bpi_T^0\bw )\bn_{TF},
      \\
      &\eqcolon
      \mathfrak{T}_{2,2,1} + \mathfrak{T}_{2,2,2} + \mathfrak{T}_{2,2,3}.
    \end{aligned}
  \end{equation}
  For $\mathfrak{T}_{2,2,1}$, we first observe that $\DIV (  \bR_T^k\uline{\bz}_T \otimes \bpi_T^0\bw) = \GRAD\bR_T^k\uline{\bz}_T\bpi_T^0\bw + \cancel{\bR_T^k\uline{\bz}_T(\DIV\bpi_T^0\bw)}\in\Poly{k}(\Ths[T])^3$.
  Hence, using H\"older inequalities with exponents $(4,2,4)$, we infer that
  \begin{align}
    |\mathfrak{T}_{2,2,1}|
    &\leq
    \sum _{T\in \calT_h}
    \sum _{\tau \in \Ths[T]}
    \norm{L^4(\tau)^3}{\bw - {\bpi}_T^k\bw}
    \norm{L^2(\tau)^{3\times3}}{\GRAD\bR_T^k\uline{\bz}_T}
    \norm{L^4(\tau)^3}{\bpi_T^0\bw}\notag\\
    &\leq
    \sum _{T\in \calT_h}
    \norm{L^4(T)^3}{\bw - {\bpi}_T^k\bw}
    \norm{L^4(T)^3}{\bpi_T^0\bw}
    \sum _{\tau \in \Ths[T]}
    \norm{L^2(\tau)^{3\times3}}{\GRAD\bR_T^k\uline{\bz}_T}\notag\\
    &\lesssim
    \sum _{T\in \calT_h}
    h_T^{k+1}\seminorm{W^{k+1,4}(T)^3}{\bw}
    \norm{W^{1,4}(T)^3}{\bw}
    \sum _{\tau \in \Ths[T]}
    \norm{L^2(\tau)^{3\times3}}{\GRAD\bR_T^k\uline{\bz}_T},
    \label{ineq:I.2.2.1}
  \end{align}
  where, in the second step, we have used the fact that $\tau\subset T$ for all $\tau \in \Ths[T]$,
  while, in the third step, we have used the approximation properties \eqref{eq:l2proj:error:cell} of the $L^2$-orthogonal projector with $(l,m,r,s)=(k,0,4,k+1)$ for the first factor and its boundedness for the second factor.
  To bound $\norm{L^2(\tau)^{3\times3}}{\GRAD\bR_T^k\uline{\bz}_T}$, we first observe that $(\bR_T^k\uline{\bz}_T) _{|\tau}$ is
  in the space $\Poly{k+1}(\tau)^3$, and it holds that
  \begin{align*}
  \norm{L^2(\tau)^{3\times3}}{\GRAD\bR_T^k\uline{\bz}_T}
  &\le
  \norm{L^2(\tau)^{3}}{\GRAD(\bR_T^k\uline{\bz}_T-\bz_T)}
  + \norm{L^2(\tau)^{3}}{\GRAD\bz_T}
  \\
  &\lesssim
  h_\tau^{-1}\norm{L^2(\tau)^{3}}{\bR_T^k\uline{\bz}_T-\bz_T}
  + \norm{L^2(\tau)^{3}}{\GRAD\bz_T}
  \\
  &\lesssim
  h_\tau^{-1}\norm{L^2(T)^{3}}{\bR_T^k\uline{\bz}_T-\bz_T}
  + \norm{L^2(T)^{3}}{\GRAD\bz_T}
  \\
  &\lesssim
  h_\tau^{-1}h_T\seminorm{1,\partial T}{\uline{\bz}_T}
  + \norm{L^2(T)^{3}}{\GRAD\bz_T}
  \lesssim
  \norm{1,T}{\uline{\bz}_T},
  \end{align*}
  where  we have started with a triangle inequality after inserting $\pm\GRAD\bz_T$,
  used a local discrete inverse inequality on $\tau$ in the second step,
  the fact that $\tau\subset T$ for all $\tau \in \Ths[T]$ in the third step,
  the bound \eqref{ineq:rtn:bound} in the fourth step,
  and the inequality $h_\tau^{-1}h_T\lesssim  1 $ valid for regular mesh sequences (see \cite[Eq. (1.4)]{Di-Pietro.Droniou:20}),
  along with the definition \eqref{eq:norm.1T} of the $\norm{1,T}{{\cdot}}$-norm to conclude.
  Plugging this last inequality into \eqref{ineq:I.2.2.1} and using the geometric bound \eqref{ineq:card.IT.F} on $\Ths[T]$ along with a discrete H\"older inequality, we arrive at
  \begin{equation}\label{eq:curl:T221}
    |\mathfrak{T}_{2,2,1}|
    \lesssim
    h^{k+1}|\bw |_{W^{k+1,4}(\calT_h)^3}   \| \bw  \|_{W^{1,4}(\Omega)^3} \| \uline{\bz}_h \|_{1,h}.    
  \end{equation}

  To estimate $\mathfrak{T}_{2,2,2}$ in \eqref{eq:curl:T22}, we insert $\pm \bz_T$ into the first factor inside the jump operator to write
  \begin{align*}
    \mathfrak{T}_{2,2,2}
    =&
    \sum _{T \in \Th}
    \sum _{\tau \in \Ths[T]}\sum_{\sigma\in\Fhs[T]^{\rm i}}\int_\sigma
    (\bw - {\bpi}_T^k\bw ) \cdot  \llbracket (\bR_T^k\uline{\bz}_T - \bz_T)\otimes \bpi_T^0\bw\rrbracket_\sigma \bn_\sigma\\
    &-
    \sum _{T \in \Th}\sum _{\tau \in \Ths[T]}\sum_{\sigma\in\Fhs[T]^{\rm i}}\int_\sigma
    \cancel{(\bw - {\bpi}_T^k\bw ) \cdot  \llbracket \bz_T\otimes \bpi_T^0\bw\rrbracket_\sigma \bn_\sigma},
  \end{align*}
  where the second addend cancels since $\bz_T\otimes \bpi_T^0{\bw}$ is continuous across the interior faces of $\Ths[T]$.
  Setting $ \Fhs^{\rm i}\coloneq \{\sigma \in \Fhs[T]^{\rm i}: T \in \Th\}$ and exchanging the order of the sums, we can now express $\mathfrak{T}_{2,2,2}$ in the following equivalent form:
  \begin{align}
    \mathfrak{T}_{2,2,2}
    =&
    \sum_{\sigma\in\Fhs^{\rm i}}
    \sum_{i=1}^2
    \int_\sigma
    (\bw - {\bpi}_{T_\sigma}^k\bw ) \cdot  \left[(\bR_{T_\sigma}^k\uline{\bz}_{T_\sigma} - \bz_{T_\sigma})\otimes \bpi_{T_\sigma}^0\bw\right]_{|\tau_i}\bn_{\tau_i\sigma},
    \label{eq:curl:T222}
  \end{align}
  where, for a given $\sigma\in\Fhs^{\rm i}$,
  $T_\sigma\in\Th$ is the element in which $\sigma$ is contained, while
  $\tau_1$ and $\tau_2$ denote the simplices in $\Ths[T_\sigma]$ sharing $\sigma$.
  To bound the right-hand side of the above expression, we apply H\"older inequalities with exponents $(4,2,4,\infty)$ along with $\|\bn_{\tau_i\sigma}\|_{L^\infty(\sigma)^3}=1$ to write
  \begin{equation}\label{eq:curl:T222:L1}
    \begin{aligned}
      |\mathfrak{T}_{2,2,2}|
      &\leq
      \sum_{\sigma\in\Fhs^{\rm i}}
      \sum_{i=1}^2
      \norm{L^4(\sigma)^3}{\bw - {\bpi}_{T_\sigma}^k\bw } \norm{L^2(\sigma)^3}{(\bR_{T_\sigma}^k\uline{\bz}_{T_\sigma})_{|\tau_i}  - \bz_{{T_\sigma}}} \norm{L^4(\sigma)^3}{\bpi_{T_\sigma}^0\bw}
      \\
      &\lesssim
      \sum_{\sigma\in\Fhs^{\rm i}}
      \sum_{i=1}^2
      \norm{L^4(\sigma)^3}{\bw - {\bpi}_{T_\sigma}^k\bw }h_{\tau_i}^{-\frac{1}{2}} \norm{L^2(\tau_i)^3}{ \bR_{T_\sigma}^k\uline{\bz}_{T_\sigma} - \bz_{T_\sigma}} h_{\tau_i}^{-\frac{1}{4}}\norm{L^{4}(\tau_i)^3}{\bpi_{T_\sigma}^0\bw}
      \\
      &\lesssim
      \sum_{\sigma\in\Fhs^{\rm i}}
      \left( h_{T_\sigma}^{\frac{1}{4}} \norm{L^4(\sigma)^3}{\bw - {\bpi}_{T_\sigma}^k\bw }\right) \norm{1,T_\sigma}{\uline{\bz}_{T_\sigma}} \norm{W^{1,4}(T_\sigma)^3}{\bw},    
    \end{aligned}
  \end{equation}
  where, in the second step, we have used local trace inequalities on the submesh for the second and third factors while, in the third step, we have used
  $\tau_i\subset T_\sigma$ along with \eqref{ineq:rtn:bound}, \eqref{eq:norm.1T}, and $h_{\tau_i}^{-1}h_{T_\sigma}\lesssim  1$ (consequence of mesh regularity) for the second factor
  while, for the third factor, we have  used again $\tau_i\subset T_\sigma$ along with the boundedness of the $L^2$-orthogonal projector.
  Using trace inequalities on the submesh along with the approximation properties of the $L^2$-orthogonal projector, we infer $h_{T_\sigma}^{\frac{1}{4}} \norm{L^4(\sigma)^3}{\bw - {\bpi}_{T_\sigma}^k \bw }\lesssim h_{T_\sigma}^{k+1}\seminorm{W^{k+1,4}(T_\sigma)^3}{\bw}$ which, plugged into \eqref{eq:curl:T222:L1} and combined with the geometric bound \eqref{ineq:card.IT.F}, gives
  \begin{align}
    \label{eq:curl:T222.bound}
    |\mathfrak{T}_{2,2,2}|\lesssim
    h^{k+1}|\bw |_{W^{k+1,4}(\calT_h)^3} \| \bw  \|_{W^{1,4}(\Omega)^3} \| \uline{\bz}_h \|_{1,h}.
  \end{align}
  
  To bound the  term $\mathfrak{T}_{2,2,3}$ in \eqref{eq:curl:T22}, we first insert $\pm {\bz}_T$ into its second factor to write
  \begin{align*}
    \mathfrak{T}_{2,2,3}
    =&
    \sum _{T\in \calT_h}\sum _{F \in \calF_T}
    \int _{F} ( \bw - {\bpi}_F^k\bw) \cdot\left[
      (\bR_T^k\uline{\bz}_T -{\bz}_T) \otimes  \bpi_T^0\bw
      \right]\bn_{TF}\\
    &+
    \sum _{T\in \calT_h}\sum _{F \in \calF_T}
    \int _{F}\cancel{ ( \bw - {\bpi}_F^k\bw) \cdot ( \bz_T \otimes  \bpi_T^0\bw )\bn_{TF}},
  \end{align*}
  where the second addend cancels by the definition \eqref{eq:l2proj:def} of the $L^2$-orthogonal projector $\bpi_F^k$ since $(\bz_T\otimes\bpi_T^0)_{|F}\bn_{TF}\in \Poly{k}(F)^3$.
  Now, we rewrite the equation above as 
  \begin{align*}
    \mathfrak{T}_{2,2,3}
    =&
    \sum _{T\in \calT_h}\sum _{F \in \calF_T}\sum _{\sigma \in \Fhs[F]}
    \int _{\sigma} ( \bw - {\bpi}_F^k\bw) \cdot\left[
      (\bR_T^k\uline{\bz}_T -{\bz}_T) \otimes  \bpi_T^0\bw
      \right]\bn_{TF},
  \end{align*}
  thus, using a similar procedure as for \eqref{eq:curl:T222}--\eqref{eq:curl:T222:L1}, but for $\sigma \in \Fhs[F]$ and $\tau_i=\tau_\sigma$  where $\tau_\sigma \in\Ths[T]$ is the simplicial subelement containing $\sigma$, we infer that
  \[
  |\mathfrak{T}_{2,2,3}|
  \lesssim
  \sum _{T\in \calT_h}\sum _{F \in \calF_T}\sum _{\sigma \in \Fhs[F]}
  \left( h_T^{\frac{1}{4}} \norm{L^4(\sigma)^3}{\bw - {\bpi}_F^k\bw }\right) \norm{1,T}{\uline{\bz}_T} \norm{W^{1,4}(T)^3}{\bw},
  \]
  in addition, using  the fact that
  \begin{align*}
    \| \bw - {\bpi}_F^k\bw \|_{L^4(\sigma)^3}
    &\leq  \| \bw - {\bpi}_F^k\bw \|_{L^4(F)^3}\\
    &\leq \| {\bpi}_T^k\bw - {\bpi}_F^k\bw \|_{L^4(F)^3}  +  \| \bw - {\bpi}_T^k\bw \|_{L^4(F)^3}\\
    &\leq \| {\bpi}_F^k({\bpi}_T^k\bw - \bw) \|_{L^4(F)^3}  +  \| \bw - {\bpi}_T^k\bw \|_{L^4(F)^3}
    \lesssim \| \bw  - {\bpi}_T^k\bw \|_{L^4(F)^3},
  \end{align*}
  the approximation properties \eqref{eq:l2proj:error:faces} of the $L^2$-orthogonal projector with $(l,m,r,s)=(k,0,4,k+1)$, and
  the bound \eqref{ineq:card.IT.F}, we obtain
  \begin{equation}\label{eq:curl:T223.bound}
    |\mathfrak{T}_{2,2,3}|\lesssim h^{k+1} |\bw |_{W^{k+1,4}(\calT_h)^3}   \| \bw  \|_{W^{1,4}(\Omega)^3} \| \uline{\bz}_h \|_{1,h}.
  \end{equation}
  
  Plugging the estimates  \eqref{eq:curl:T221}, \eqref{eq:curl:T222.bound}, and  \eqref{eq:curl:T223.bound} into \eqref{eq:curl:T22},
%%   we conclude that
%%   \[
%%   |\mathfrak{T}_{2,2}|\lesssim h^{k+1}\seminorm{W^{k+1,4}(\calT_h)^3}{\bw} \norm{W^{1,4}(\Omega)^3}{\bw} \norm{1,h}{\uline{\bz}_h},
%%   \]
%%   which, combined
  and, combining the resulting estimate  with \eqref{eq:curl:T21}, we finally obtain
  \begin{equation}
    \label{eq:curl:T2:bound}
    |\mathfrak{T}_{2}|\lesssim h^{k+1}|\bw |_{W^{k+1,4}(\calT_h)^3}   \| \bw  \|_{W^{1,4}(\Omega)^3} \| \uline{\bz}_h \|_{1,h}.
  \end{equation}
  \\
  \underline{(iii.C) \emph{Estimate of $\mathfrak{T}_3$ and $\mathfrak{T}_4$.}}
  To bound $\mathfrak{T}_3$, we follow the same steps as in \cite[Lemma 7.iii.C]{Castanon-Quiroz.Di-Pietro:20}
  along with   the boundedness \eqref{eq:disc:grad:bound} of  $\bG_{\Ths[T]}^{2(k+1)}$
  to obtain
  $$ %% \begin{equation}\label{eq:curl:T3:bound:p}
  |\mathfrak{T}_3|
  \lesssim  
  |\bw|_{H^1(\Omega)^3}\left(
  \sum_{T\in\calT_h}\| \bw- \bR_T^k\uline{\hat{\bw}}_T\|_{L^4(T)^3}^4
  \right)^{\frac14}
  \| \uline{\bz}_h \|_{1,h}.
  $$ %% \end{equation}
  To estimate each addend in the second factor, we first insert $\pm \bpi_T^k\bw$  and then use a triangle inequality to write 
  \begin{equation}\label{eq:IRTN:l4:ebound1}
    \begin{aligned}
      \| \bw- \bR_T^k\uline{\hat{\bw}}_T\|_{L^4(T)^3}
      &\leq
      \| \bw- \bpi_T^k\bw \|_{L^4(T)^3}
      + \| \bpi_T^k\bw - \bR_T^k\uline{\hat{\bw}}_T \|_{L^4(T)^3}
      \\
      &\lesssim h_T^{k+1}|\bw |_{W^{k+1,4}(T)^3}
      + \| \bpi_T^k\bw - \bR_T^k\uline{\hat{\bw}}_T \|_{L^4(T)^3},
    \end{aligned}
  \end{equation}
  where we have used the approximation properties  \eqref{eq:l2proj:error:cell} of $\bpi_T^k$  with $(l,m,r,s)=(k,0,4,k+1)$ to conclude.
  To estimate the second term in the right-hand side of \eqref{eq:IRTN:l4:ebound1}, we proceed as follows:
  \begin{align*}  
    \| \bpi_T^k\bw - \bR_T^k\uline{\hat{\bw}}_T \|_{L^4(T)^3}^4
    &=  \sum_{\tau\in\Ths[T]}  \| \bpi_T^k\bw - \bR_T^k\uline{\hat{\bw}}_T \|_{L^4(\tau)^3}^4
    \\
    &\lesssim
    \sum_{\tau\in\Ths[T]} 
    \left(h_\tau^{-\frac{3}{4}}\| \bpi_T^k\bw - \bR_T^k\uline{\hat{\bw}}_T \|_{L^2(\tau)^3}\right)^4\\
    &\leq
    \sum_{\tau\in\Ths[T]}
    h_\tau^{-3}
    \left(
    \| \bpi_T^k\bw - \bw \|_{L^2(T)^3}
    + \| \bw - \bR_T^k\uline{\hat{\bw}}_T  \|_{L^2(T)^3}
    \right)^4\\
    \\
    &\lesssim
    \sum_{\tau\in\Ths[T]}
    h_\tau^{-3}  h_T^{4(k+1)} |\bw|_{H^{k+1}(T)^3}^4
    \lesssim
    h_T^{4(k+1)}|\bw |_{W^{k+1,4}(T)^3}^4,
  \end{align*}
  where:
  to pass to the second line we have used the reverse Lebesgue embedding \eqref{eq:rev:emb:T} for $(X,\alpha,\beta)=(T,4,2)$;
  to pass to the third line, we have inserted $\pm\bw$ and used a triangle inequality along with $\tau\subset T$;
  to pass to the fourth line, we have used the approximation properties  \eqref{eq:l2proj:error:cell} of the $L^2$-orthogonal projector with $(l,m,r,s)=(k,0,2,k+1)$ for the first addend and
  the approximation property \eqref{ineq:rtn:approx} for the second addend;
  the conclusion follows from $h_\tau^{-1} h_T \lesssim 1$ (consequence of mesh regularity), the bound \eqref{ineq:card.IT.F} on $\card(\Ths[T])$, and the Lebesgue embedding $\|\zeta\|_{L^2(T)}\lesssim h_T^{\frac{3}{4}}  \|\zeta\|_{L^4(T)}$ valid for all $\zeta \in L^4(T)$.
  Plugging the above bound into \eqref{eq:IRTN:l4:ebound1}, we get
  $$
  \| \bw- \bR_T^k\uline{\hat{\bw}}_T\|_{L^4(T)^3}
  \lesssim    h_T^{k+1}|\bw |_{W^{k+1,4}(T)^3}.
  $$
  In conclusion, we have that         
  \begin{align}\label{eq:curl:T3:bound}
    |\mathfrak{T}_3|
    &\lesssim
    h^{k+1}| \bw |_{H^1(\Omega)^3} | \bw  |_{W^{k+1,4}(\calT_h)^3} \| \uline{\bz}_h \|_{1,h}.
  \end{align}
  Using similar arguments as for $\mathfrak{T}_3$, we have for the last term
  \begin{equation}\label{eq:curl:T4:bound}
    |\mathfrak{T}_4|
    \lesssim
    h^{k+1}| \bw |_{H^1(\Omega)^3}  | \bw |_{W^{k+1,4}(\calT_h)^3} \| \uline{\bz}_h \|_{1,h}.
  \end{equation}
  \\
  \underline{(iv.D) \emph{Conclusion.}}
  Taking absolute values in \eqref{eq:curl:all:T},
  recalling the definition \eqref{eq:dual.norm} of  the dual norm, and using the estimates
  \eqref{eq:curl:T1:bound},
  \eqref{eq:curl:T2:bound},
  \eqref{eq:curl:T3:bound},
  and \eqref{eq:curl:T4:bound}, and additionally noticing that 
  $\seminorm{H^1(\Omega)^3}{\bw}\lesssim\seminorm{W^{1,4}(\Omega)^3}{\bw}$, the conclusion follows.
\end{proof}

\subsection{Discrete problem and main results}\label{sec:disc.problem}

The HHO discretization of problem \eqref{eq:nstokes:weak} reads:
Find $(\uline{\bu}_h,p_h)\in \uline{\bU}_{h,0}^k\times P_h^k$ such that
\begin{subequations}
  \label{eq:nstokes:discrete}
  \begin{alignat}{2}
    \nu a_h(\uline{\bu}_h,\uline{\bv}_h)
    +
    t_h(\uline{\bu}_h,\uline{\bu}_h,\uline{\bv}_h)
    +b_h(\uline{\bv}_h,p_h)  
    &= \ell_h(\bef,\uline{\bv}_h)  &\qquad&\forall \uline{\bv}_h \in \uline{\bU}^{k}_{h,0}
,
  \label{eq:nstokes:discrete:momentum}
    \\
    -b_h(\uline{\bu}_h,q_h)&=0 &\qquad&\forall q_h \in  \mathbb{P}^k(\calT_h).
    \label{eq:nstokes:discrete:mass}
  \end{alignat}
\end{subequations}
The existence of a solution to \eqref{eq:nstokes:discrete} for any $\bef\in L^2(\Omega)^3$ can be proved using a topological degree argument as in \cite[Theorem 1]{Di-Pietro.Krell:18}.
Similarly, uniqueness can be proved along the lines of Theorem 2 therein under a smallness condition on $\bef$.

Recalling the Hodge decomposition \eqref{eq:hodge.f} and denoting by $C_P$ a Poincar\'e constant in $\Omega$, Proposition \ref{lem:uh:a-priori} below is the discrete equivalent of the following a priori continuous bound (see \cite[Section 2.3]{Castanon-Quiroz.Di-Pietro:20})
\begin{equation}\label{eq:u:a-priori}
  |\bu|_{H^1(\Omega)^3}\le \nu^{-1}C_{\rm P}\|\bg\|_{L^2(\Omega)^3}.
\end{equation}

\begin{proposition}[Uniform a priori bound on the discrete velocity]\label{lem:uh:a-priori}
  Let $(\uline{\bu}_h,p_h)\in\uline{\bU}_{h,0}^k\times P_h^k$ be a solution to \eqref{eq:nstokes:discrete}.
  Then, given the Hodge decomposition \eqref{eq:hodge.f} of $\bef$, we have the following uniform a priori bound for the velocity:
  \[%% begin{equation}\label{eq:uh:a-priori}
    \|\uline{\bu}_h\|_{1,h}\lesssim\nu^{-1}\|\bg\|_{L^2(\Omega)^3}.
  \]%% end{equation}
\end{proposition}
\begin{proof}
The proof follows the same reasoning  as \cite[Proposition 8]{Castanon-Quiroz.Di-Pietro:20} with Lemmas \ref{lem:ns:lh} and \ref{lem:ns:th} replacing, respectively, \cite[Eqs. (41)--(42) and Lemma 7]{Castanon-Quiroz.Di-Pietro:20}.
\end{proof}

\begin{remark}[Efficient implementation]\label{rem:static.cond}
  When solving the algebraic problem corresponding to \eqref{eq:nstokes:discrete} by a first order iterative algorithm, all element-based velocity unknowns and all but one pressure unknowns per element can be statically condensed at each iteration in the spirit of \cite[Section 6.2]{Di-Pietro.Ern.ea:16}
    ; see \cite{Botti.Di-Pietro:22} for a study of the effect of static condensation strategies on the multigrid resolution of the global algebraic systems arising from HHO discretizations of incompressible flow problems.
\end{remark}
 
\begin{remark}[The two-dimensional case]\label{rem:2d}
  The two-dimensional version of the method \eqref{eq:nstokes:discrete} will be considered numerically in Section \ref{sec:tests}.
  Denoting by $u_i$, $i=1,\ldots,3$, the component of the velocity field along the Cartesian axis $x_i$, the two-dimensional plane velocity problem can be recovered from \eqref{eq:nstokes:weak} setting $u_3=0$ and assuming that $u_1$ and $u_2$ do not depend on $x_3$.
\end{remark}

We next consider the discretization error defined as the difference between the solution to the HHO scheme and the interpolate of the exact solution.

\begin{theorem}[Error estimate for small data]\label{thm:nstokes:disc:err.est}
  Recalling the Hodge decomposition \eqref{eq:hodge.f} of the forcing term $\bef$, we assume that it holds, for some $\alpha\in (0,1),$
  \begin{equation*}%% \label{eq:smalldata:assum}
    \|\bg \|_{\Ldeuxd}\leq \alpha\frac{\nu^2C_a}{C_tC_IC_P},
  \end{equation*}
  where $C_a$ and $C_t$ are defined in \eqref{eq:ns:stab.h} and \eqref{ns:th:boundedness}, while $C_I$ denotes the continuity constant of the HHO interpolator in the discrete $H^1$-like norm (see \cite[Proposition 2.2]{Di-Pietro.Droniou:20}) and $C_P$ is the Poincar\'e constant in \eqref{eq:u:a-priori}.  
{Let $k\geq 0$} and let $(\bu,p) \in \bU\times P$ and $(\underline{\bu}_h,p_h) \in  \uline{\bU}_h^k \times P_h^k $ solve \eqref{eq:nstokes:weak} and \eqref{eq:nstokes:discrete}, respectively.
  Assuming the additional regularity $\bu\in H^{k+2}(\calT_h)^3$ and $p\in H^1(\Omega)\cap H^{k+1}(\calT_h)$, it holds:
  \begin{multline}\label{eq:nstokes:disc:err.est}
    \| \uline{\bu}_h - \uline{\bI}_h^k\bu \|_{1,h}
    + \nu^{-1}  \| p_h - \pi_h^kp \|_{\Ldeux}
    \\
    \lesssim
    h^{k+1}(1-\alpha)^{-1} \left (
    |\bu|_{H^{k+2}(\calT_h)^3} + {\nu}^{-1} \|\bu\|_{W^{1,4}(\Omega)^3}  | \bu |_{W^{k+1,4}(\calT_h)^3}
    \right).
  \end{multline}
  where the hidden constant is independent of $\nu$, $\lambda$, $h$, as well as $(\bu,p)$.
\end{theorem}

\begin{proof}
  Analogous to that of \cite[Theorem 11]{Castanon-Quiroz.Di-Pietro:20}.
\end{proof}

\begin{remark}[Pressure robustness]\label{rem:pressure-robustness}
  The error estimate \eqref{eq:nstokes:disc:err.est} is pressure-robust since the right-hand side does not depend on $\lambda$ in \eqref{eq:hodge.f} nor on the pressure.
\end{remark}

%------------------------------------------------------------------------------%
%------------------------------------------------------------------------------%

\section{Numerical tests}\label{sec:tests}

In this section we  verify  numerically the proposed method for general meshes with convex elements for $\Omega \subset \mathbb{R}^2$. For each element $T \in \Th$, we construct its simplicial submesh $\Ths[T]$ {using an ear clipping algorithm}, i.e.,
we construct $\Ths[T]$
 in such a way that no additional  internal nodes  are introduced and that $\Fhs[T] = \Fh[T]$ (this construction fulfills the assumptions made in Section \ref{sec:setting:mesh}).
For the sake of completeness, we also include comparisons with the original HHO method  of \cite{Botti.Di-Pietro.ea:19}.
Our implementation is based on the \texttt{HArDCore} library\footnote{\url{https://github.com/jdroniou/HArDCore}} and makes extensive use of the linear algebra \texttt{Eigen} open-source library \cite{Guennebaud.Jacob.ea:10}.
All the steady-state computations presented hereafter are done by means of the pseudo-transient-continuation algorithm analyzed by \cite{Kelley.Keyes:98} employing the Selective Evolution Relaxation (SER) strategy \cite{Mulder.Van-Leer:85} for evolving the pseudo-time step according to the Newton's equations residual.
Convergence to steady-state is attained when the Euclidean norm of the residual for the momentum equation drops below $10^{-11}$.
At each pseudo-time step, the linearized equations are exactly solved by means of the direct solver Pardiso \cite{Schenk.Gartner.ea:01}.
Accordingly, the Euclidean norm of the residual for the continuity equation is comparable to the machine epsilon at all pseudo-time steps.

\begin{figure}[!htb]
  \centering
  \begin{minipage}[b]{0.32\textwidth}
    \centering
    \includegraphics[width=0.80\textwidth]{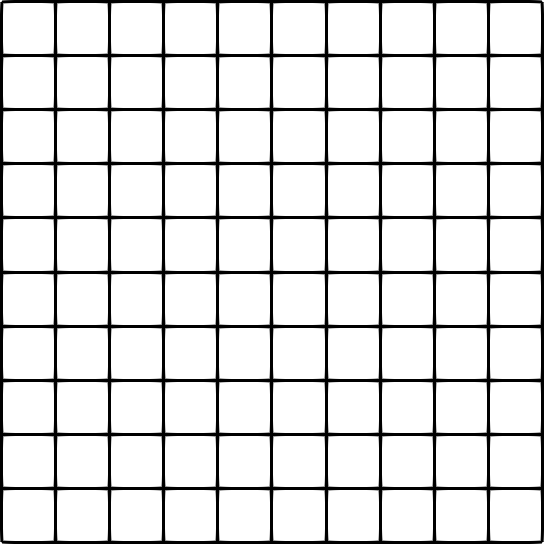} 
    \subcaption{Cartesian.}
  \end{minipage}%
  \begin{minipage}[b]{0.32\textwidth}
    \centering
    \includegraphics[width=0.80\textwidth]{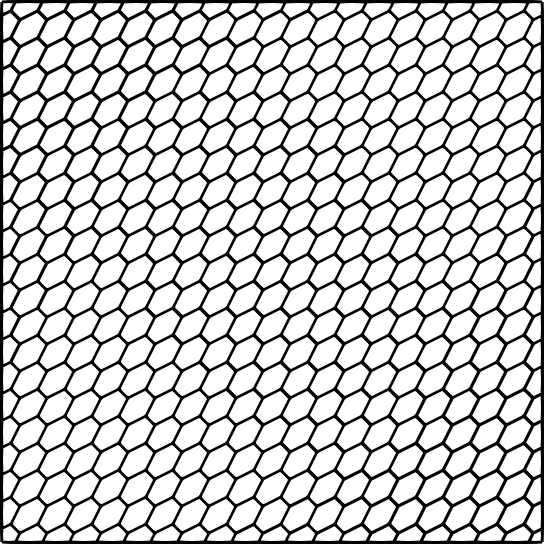} 
    \subcaption{Hexagonal.}
  \end{minipage}%
  \begin{minipage}[b]{0.32\textwidth}
  \centering
  \includegraphics[width=0.80\textwidth]{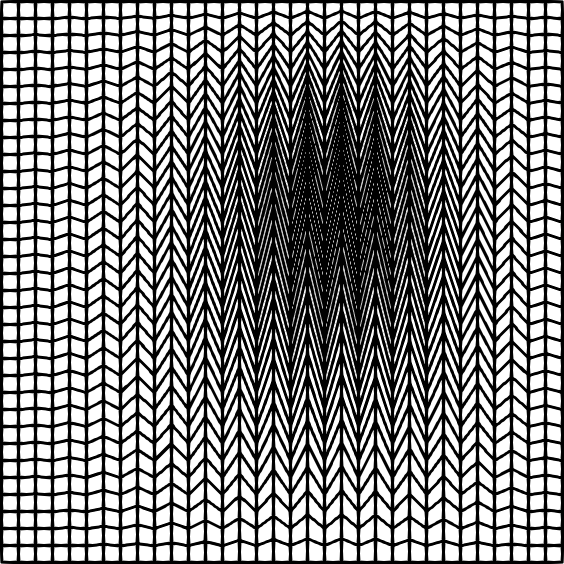} 
  \subcaption{Kershaw.}
  \end{minipage}%
  \caption{Coarsest meshes used in Section \ref{sec:tests:kovasznay}. \label{fig:meshes:coarsest}}
\end{figure}

%------------------------------------------------------------------------------%
% Kovasznay
%------------------------------------------------------------------------------%

\subsection{Kovasznay flow}\label{sec:tests:kovasznay}
We start by assessing the convergence properties of the method using the well known analytical solution of Kovasznay \cite{Kovasznay:48} with $\nu=0.025$; see, e.g., \cite[Section 6.1]{Di-Pietro.Droniou:20} for the expression of the velocity and pressure fields.
We consider computations  over three $h$-refined mesh families (Cartesian, hexagonal and Kershaw type). 
Figure \ref{fig:meshes:coarsest} shows the coarsest mesh for each family.
We monitor the following quantities in Table \ref{table:kovasznay}:
$N_{\rm dof}$ and $N_{\rm nz}$ denoting, respectively, the number of discrete unknowns and nonzero entries of the statically condensed linearized problem;
$\nnorm[\nu,h]{\uline{\be}_h}\coloneq\left[\nu\mathrm{a}_h(\uline{\be}_h,\uline{\be}_h)\right]^{\nicefrac12}$, the energy norm of the error $\uline{\be}_h\coloneq\uline{\bu}_h - \uline{\bI}_h^k\bu$ on the velocity 
(using  the global norm equivalence \eqref{eq:ns:stab.h}, 
an estimate in $h^{k+1}$ for this quantity is readily inferred from \eqref{eq:nstokes:disc:err.est});
$\nnorm[L^2(\Omega)^2]{\be_h}$ and $\nnorm[L^2(\Omega)]{\epsilon_h}$, the $L^2$-errors on the velocity and the pressure, respectively.
Each error measure is accompanied by the corresponding Estimated Order of Convergence (EOC) computed using successive refinement steps.
The results collected in Table \ref{table:kovasznay} show that both the energy norm of the error on the velocity and the $L^2$-norm of error on the pressure converge as $h^{k+1}$ as expected. Additionally, the $L^2$-norm of the error of the velocity converges with rates close to $h^{k+2}$. 
%%%%%%%%%%%%%%%%%%%%%%%%%%%%%%%%%%%%%%%%%%%%%%%%%%%%%%%%%%%%%%%%%%%%%%%%%%%%%%%%%%%%%%%%%%%%%%
%Convergence rates for kovasznay
\renewcommand\floatplace[1]{}
\begin{table}
  \centering
  \begin{tabular}{ccccccc}
    \toprule
    $N_{\rm dof}$      &
    $\nnorm[\nu,h]{\uline{\be}_h}$& EOC
    & $\nnorm[L^2(\Omega)^2]{{\be}_h}$    & EOC
    & $\nnorm[L^2(\Omega)]{\epsilon_h}$  &EOC \\
    \midrule
    \multicolumn{7}{c}{Cartesian, $k=0$} \\
    \midrule
    540           &  5.92E-01  & --          &  1.11E-01  & --           &  2.54E-01  & --           \\
    2080          & 3.35E-01   & 0.822       & 3.54E-02   & 1.644        & 8.58E-02   & 1.565        \\
    8160         & 1.78E-01   & 0.911       & 9.92E-03   & 1.835        & 2.46E-02   & 1.800        \\
    32320         & 9.10E-02   & 0.969       & 2.59E-03   & 1.938        & 6.50E-03   & 1.923        \\
    \midrule
    \multicolumn{7}{c}{Cartesian, $k=1$} \\
    \midrule
    980     & 2.04E-01      & --       & 1.97E-02    &--       & 4.91E-02    & --      \\
    3760    & 5.73E-02      & 1.831    & 2.28E-03    & 3.114   & 5.86E-03    & 3.067   \\
    14720   & 1.51E-02      & 1.926    & 2.91E-04    & 2.970   & 7.75E-04    & 2.918   \\
    58240   & 3.85E-03      & 1.96     & 3.75E-05    & 2.957   & 1.08E-04  &  2.845    \\
    \midrule
    \multicolumn{7}{c}{Hexagonal, $k=0$} \\
    \midrule
    3241   &  8.26E-01  &  --    &  5.46E-02  &  --    &  1.53E-01  &  --       \\
    12081  &  4.42E-01  &  0.901 &  1.71E-02  &  1.674 &  4.21E-02  &  1.861  \\
    46561  &  2.27E-01  &  0.964 &  4.78E-03  &  1.838 &  1.10E-02  &  1.929  \\
    182721 &  1.14E-01  &  0.986 &  1.25E-03  &  1.931 &  2.85E-03  &  1.954  \\
    \midrule
    \multicolumn{7}{c}{Hexagonal, $k=1$} \\
    \midrule
    6041     &  4.80E-01 &    --    &  3.04E-02 &    --    &  1.03E-01 &    --        \\
    22481     &  8.40E-02 &    2.515 &  1.42E-03 &    4.421 &  4.53E-03 &    4.507   \\
    86561     &  1.78E-02 &    2.237 &  1.23E-04 &    3.529 &  4.20E-04 &    3.432   \\
    339521     &  4.20E-03 &    2.085 &  1.37E-05 &    3.164 &  6.84E-05 &    2.618   \\
    \midrule
    \multicolumn{7}{c}{Kershaw, $k=0$} \\
    \midrule
    5577   & 5.76E-01  &   --       & 1.78E-01  &   --        & 2.10E-01  &   --        \\
    22044  & 2.46E-01  &   1.231    & 6.36E-02  &   1.488     & 8.74E-02  &   1.269   \\
    49401  & 1.50E-01  &   1.230    & 2.93E-02  &   1.921     & 4.27E-02  &   1.774   \\
    87648  & 1.08E-01  &   1.142    & 1.66E-02  &   1.983     & 2.47E-02  &   1.908   \\
    136785 & 8.46E-02  &   1.092    & 1.06E-02  &   1.995     & 1.60E-02  &   1.951   \\
    \midrule
    \multicolumn{7}{c}{Kershaw, $k=1$} \\
    \midrule
    10065 & 4.51E-01  & --           & 5.88E-02  & --           & 1.50E-01  & --         \\
    39732   & 7.00E-02  &  2.701       & 2.35E-03  &  4.672       & 5.12E-03  &  4.897     \\
    89001   & 3.02E-02  &  2.080       & 5.08E-04  &  3.785       & 1.09E-03  &  3.830     \\
    157872  & 1.64E-02  &  2.131       & 1.78E-04  &  3.653       & 4.23E-04  &  3.299     \\
    246345  & 1.04E-02  &  2.032       & 7.80E-05  &  3.703       & 2.11E-04  &  3.112     \\
    \bottomrule
  \end{tabular}
\caption{Convergence rates for the numerical test of Section \ref{sec:tests:kovasznay}.\label{table:kovasznay}}
\end{table}

\subsection{Robustness of the velocity error estimate}\label{ssec:lif:test}

The second numerical example, inspired by \cite[Benchmark 3.3]{Linke.Merdon:16*1}, is meant to demonstrate the robustness of the proposed method for large irrotational body forces.
Specifically, we verify numerically the fact that the approximation of the velocity is independent of both $\lambda$ and $p$.
Letting $\Omega=(0,1)^2$ and $\lambda\ge 0$, we solve the Dirichlet problem corresponding to the exact solution $(\bu,p)$ in \eqref{eq:nstokes:weak} with velocity components given by $\bu(\bx) \coloneq \begin{pmatrix}-x_2 \\ x_1\end{pmatrix}$ and pressure given by $p(\bx) \coloneq {\lambda}x_1^3 + \frac{x_1^2+ x_2^2}{2} -\frac{1}{4}$.
We set $\nu=1$, then observe that the force in \eqref{eq:nstokes:weak:momentum} is purely irrotational, i.e., $\bef(\bx) = \begin{pmatrix} 3\lambda x_1^2 \\ 0\end{pmatrix}$.
In the computations, we take $\lambda=10^6$ and consider a sequence of uniformly $h$-refined meshes equivalent (by scaling and translation) to the three mesh families used in the previous section, see Figure \ref{fig:meshes:coarsest}. 
Table \ref{tbl:lif:prf:cart.hex} collects the results for the Cartesian and hexagonal mesh families, and Table \ref{tbl:lif:prf:ker} for the Kershaw mesh family.
For the sake of comparison, we also report in these tables the corresponding results obtained using the original HHO method of \cite{Botti.Di-Pietro.ea:19}.
It can be noticed that the solution is exactly reproduced by the present method with $k=1$ on all the meshes, while a quick convergence is observed for $k=0$ on the hexagonal and Kershaw meshes, most likely due to the quadratic nature of the pressure.
By contrast, the HHO method of \cite{Botti.Di-Pietro.ea:19} shows large errors on the velocity due to the lack of pressure-robustness.

%------------------------------------------------------------------------------%
% Cartesian & Hexagonal-Pressure Robust Ex
%------------------------------------------------------------------------------%

\begin{sidewaystable}
  \centering
  \begin{tabular}{ccccccc|cccccc}
    \toprule
    $N_{\rm dof}$      &
    $\nnorm[\nu,h]{\uline{\be}_h}$& EOC
    & $\nnorm[L^2(\Omega)^2]{{\be}_h}$    & EOC
    & $\nnorm[L^2(\Omega)]{\epsilon_h}$  &EOC 
    & $\nnorm[\nu,h]{\uline{\be}_h}$& EOC
    & $\nnorm[L^2(\Omega)^2]{{\be}_h}$    & EOC
    & $\nnorm[L^2(\Omega)]{\epsilon_h}$  &EOC \\
    \midrule
    \multicolumn{1}{c}{}&
    \multicolumn{6}{c|}{Cartesian, $k=0$, proposed method}&
    \multicolumn{6}{c}{Cartesian, $k=0$,  HHO method of \cite{Botti.Di-Pietro.ea:19} }\\
    \midrule
    540   &  1.94E-11   & --  &  1.74E-12   & --      &  8.28E-16   & --   
    
    & 2.77E+05  &   --    & 1.79E+04  &   --   & 1.22E+02  &   --         \\   

    2080  &  3.06E-11   & --  &  1.68E-12   & 0.052   &  1.09E-15   & --   
    & 2.34E+04  &   3.562 & 9.93E+02  &   4.170& 1.56E+00  &   6.290 \\    
    8160 &  3.61E-11   & --  &  6.22E-12   & --      &  1.88E-15   & --     
    & 1.18E+04  &   0.991 & 2.57E+02  &   1.951& 1.13E-01  &   3.789    \\  
    32320 &  3.30E-11   & 0.129  &  1.77E-12   & 1.816  &  1.01E-15   & 0.904  
    & 5.92E+03  &   0.994 & 6.54E+01  &   1.973& 8.67E-03  &   3.705       \\
    \midrule
    \multicolumn{1}{c}{}&
    \multicolumn{6}{c|}{Cartesian, $k=1$, proposed method}&
    \multicolumn{6}{c}{Cartesian, $k=1$,  HHO method of \cite{Botti.Di-Pietro.ea:19} }\\
    \midrule
    980   & 1.44E-10    & --     & 5.16E-12    & --     & 6.07E-05    & --

    &2.01E+03  &  --     &9.91E+01  &  --      &3.09E-02  &  --          \\
    3760  & 1.41E-10    & 0.034  & 3.22E-12   & 0.682   & 7.59E-06    & 3.000      
    &5.06E+02  &  1.987  &1.26E+01  &  2.977   &5.06E-04  &  5.936       \\

    14720 & 1.81E-10    & --     & 2.89E-12    & 0.155  & 9.49E-07    & 3.000      
    &1.27E+02  &  1.994  &1.58E+00  &  2.990   &8.66E-06  &  5.867       \\ 
    58240 & 1.47E-10    & 0.306  & 2.42E-12    & 0.257  & 1.19E-07    & 3.000      
    &3.18E+01  &  1.997  &1.99E-01  &  2.995   &5.06E-07  &  4.097       \\
    \midrule
    \multicolumn{1}{c}{}&
    \multicolumn{6}{c|}{Hexagonal, $k=0$, proposed method}&
    \multicolumn{6}{c}{Hexagonal, $k=0$,  HHO method of \cite{Botti.Di-Pietro.ea:19} }\\
    \midrule
    3241       &  1.06E-01  &   --     &  4.18E-03  &   --      &  5.17E-05  &   --      
    & 1.97E+04   &   --    & 7.80E+02   &   --     & 1.04E+00   &   --    \\
    12081      &  9.78E-03  &   3.434  &  2.51E-04  &   4.056   &  9.34E-06  &   2.468   
    & 1.48E+04   &   0.415 & 4.37E+02   &   0.838  & 3.33E-01   &   1.640 \\          

    46561       &  8.85E-04  &   3.466  &  1.53E-05  &   4.038   &  1.67E-06  &   2.484   
    & 8.18E+03   &   0.854 & 1.36E+02   &   1.686  & 3.52E-02   &   3.244 \\
    182721       &  7.92E-05  &   3.482  &  9.42E-07  &   4.021   &  2.97E-07  &   2.492   
    & 4.15E+03   &   0.979 & 3.51E+01   &   1.950  & 2.92E-03   &   3.591 \\
    \midrule
    \multicolumn{1}{c}{}&
    \multicolumn{6}{c|}{Hexagonal, $k=1$, proposed method}&
    \multicolumn{6}{c}{Hexagonal, $k=1$,  HHO method of \cite{Botti.Di-Pietro.ea:19} }\\
    \midrule
    6041      & 2.32E-10   & --    & 2.13E-11   & --    & 7.13E-06   & --        
    & 6.72E+02  &   --     & 1.81E+01  &   --    & 1.01E-03  &   --     \\
    22481     & 2.24E-10   & 0.050 & 1.08E-11   & 0.982 & 9.17E-07   & 2.959     
    & 1.76E+02  &   1.933  & 2.23E+00  &   3.019 & 2.74E-05  &   5.207  \\
    86561     & 2.60E-10   & --    & 1.08E-11   & 0.007 & 1.16E-07   & 2.980   
    & 4.46E+01  &   1.981  & 2.81E-01  &   2.989 & 4.00E-06  &   2.774  \\
    339521      & 6.67E-10   &  --   & 1.36E-10   &  --   & 1.46E-08   & 3.0      
    & 1.12E+01  &   1.990  & 3.53E-02  &   2.991 & 6.96E-07  &   2.525  \\
    \bottomrule
  \end{tabular}
  \caption{Convergence rates for the numerical test of Section \ref{ssec:lif:test} for $\lambda=10^6$ using the Cartesian and hexagonal mesh families.\label{tbl:lif:prf:cart.hex}}.
\end{sidewaystable}

%------------------------------------------------------------------------------%
% Kershaw-Pressure Robust Ex
%------------------------------------------------------------------------------%

\begin{sidewaystable}
    \centering
    \begin{tabular}{ccccccc|cccccc}
      \toprule
      $N_{\rm dof}$      &
      $\nnorm[\nu,h]{\uline{\be}_h}$& EOC
      & $\nnorm[L^2(\Omega)^2]{{\be}_h}$    & EOC
      & $\nnorm[L^2(\Omega)]{\epsilon_h}$  &EOC 
      & $\nnorm[\nu,h]{\uline{\be}_h}$& EOC
      & $\nnorm[L^2(\Omega)^2]{{\be}_h}$    & EOC
      & $\nnorm[L^2(\Omega)]{\epsilon_h}$  &EOC \\
      \midrule
      \multicolumn{1}{c}{}&
      \multicolumn{6}{c|}{Kershaw, $k=0$, proposed method}&
      \multicolumn{6}{c}{Kershaw, $k=0$,  HHO method of \cite{Botti.Di-Pietro.ea:19} }\\
      \midrule
      5577      &  4.55E-03  &  --        &  7.46E-04  &  --        &  2.72E-07  &  --     
      & 1.42E+04  &  --    & 3.74E+02  &  --     & 2.43E-01  &  --             \\  
      22044      &  2.83E-04  &  4.027     &  4.73E-05  &  3.999     &  1.76E-08  &  3.970  
      & 7.17E+03  &  0.991 & 9.60E+01  &  1.970  & 1.80E-02  &  3.771        \\
      49401      &  5.59E-05  &  4.015     &  9.37E-06  &  4.006     &  3.50E-09  &  3.995   
      & 4.79E+03  &  1.000 & 4.30E+01  &  1.986  & 4.32E-03  &  3.532        \\

      87648      &  1.77E-05  &  4.010     &  2.97E-06  &  4.005     &  1.11E-09  &  3.992   

      & 3.59E+03  &  1.001 & 2.43E+01  &  1.985  & 1.73E-03  &  3.187        \\
      136785     &  7.23E-06  &  4.008     &  1.22E-06  &  4.005     &  4.59E-10  &  3.967   
      & 2.87E+03  &  1.001 & 1.56E+01  &  1.986  & 9.19E-04  &  2.845        \\
      
      \midrule
      \multicolumn{1}{c}{}&
      \multicolumn{6}{c|}{Kershaw, $k=1$, proposed method}&
      \multicolumn{6}{c}{Kershaw, $k=1$,  HHO method of \cite{Botti.Di-Pietro.ea:19} }\\
      \midrule
      
      10065   &  5.35E-10    &--          &  8.69E-11    &--          &  1.69E-06    &--         

      & 1.86E+02  &  --          & 2.88E+00  &  --          & 3.56E-05  &  --        \\
      39732  &  5.45E-10    &--          &  7.84E-11    &0.150       &  2.11E-07    &3.014    

      & 4.67E+01  &  2.006       & 3.56E-01  &  3.033       & 2.31E-06  &  3.965     \\
      89001  &  1.09E-09    &--          &  1.76E-10    &--          &  6.26E-08    &3.008    
      & 2.08E+01  &  2.004       & 1.05E-01  &  3.013       & 5.99E-07  &  3.343     \\
      157872 &  1.60E-09    &--          &  2.72E-10    &--          &  2.64E-08    &3.006     
      & 1.17E+01  &  2.003       & 4.43E-02  &  3.008       & 2.36E-07  &  3.247     \\

      246345 &  4.53E-10    &5.665       &  4.84E-11    &7.756       &  1.35E-08    &3.005    
      & 7.49E+00  &  2.002       & 2.27E-02  &  3.005       & 1.16E-07  &  3.189     \\
      \bottomrule
    \end{tabular}
  \caption{Convergence rates for the numerical test of Section \ref{ssec:lif:test} for $\lambda=10^6$ using the Kershaw mesh family.\label{tbl:lif:prf:ker}}
\end{sidewaystable}        

%------------------------------------------------------------------------------%
%------------------------------------------------------------------------------%

\subsection{Two-dimensional lid-driven cavity flow}\label{sec:tests:cavity.2d}

The final numerical test is the classical two-dimensional lid-driven cavity problem.
The computational domain is  the unit square $\Omega=(0,1)^2$ and we initially set $\bef=\boldsymbol{0}$.
Homogeneous (wall) boundary conditions are enforced at all but the top horizontal wall (at $x_2=1$), where we enforce a unit tangential velocity  $\bu=(1,0)$ instead.
In Figure \ref{fig:cavity:Re1000} we report the horizontal component $u_1$ of the velocity along the vertical centerline $x_1=\frac12$ and the vertical component $u_2$ of the velocity along the horizontal centerline $x_2=\frac12$ for a global Reynolds number $\Reynolds\coloneq\frac1{\nu} = 1000$.
The computation is carried out setting $k=1$ for the finest meshes of the Cartesian, hexagonal, and Kershaw sequences used in the previous section.
Reference solutions from the literature \cite{Ghia.Ghia.ea:82,Erturk.Corke.ea:05} are also included for the sake of comparison.
The numerical solution obtained using the proposed method is in  agreement with the reference results for the selected value of the Reynolds number.

To check the robustness of the method with respect to irrotational body forces, we then run the same test case but with $\bef=\lambda \GRAD \psi$ where $\psi= \frac{1}{3}( x^3 + y^3)$.
This body force is completely irrotational, so the velocity approximation obtained using the proposed method \eqref{eq:nstokes:discrete} should not be affected (and, therefore, should not depend on $\lambda$).
To verify this, we report in Figure \ref{fig:cavity:Re1000:wlambda} computations for $\lambda=10^6$, using $k=1$ and the same meshes as before.
As expected, the velocity profiles are not affected by the value of $\lambda$.
The same plot also contains the results obtained with the original HHO formulation of \cite{Botti.Di-Pietro.ea:19}, but only for the Cartesian mesh and $\lambda=10^3$ (convergence was not achieved for $\lambda=10^6$).
 It can be checked that the non-pressure-robust version of the method converges to a complete different solution.
 
%------------------------------------------------------------------------------%
% Re = 1000
%------------------------------------------------------------------------------%

\begin{figure}\centering
  \begin{tikzpicture}[font=\footnotesize, %
      spy using outlines={magnification=4, size=3cm, connect spies, fill=none} %
    ]
    \begin{axis}[height=9cm, width=9cm, %
        xmin=-1, xmax=1, ymin=0, ymax=1, %
        xlabel={$u_1$}, ylabel={$x_2$}, %
        legend cell align={left},
        legend style = { at={(0,1)}, anchor=north west, draw=none, fill=none }, % 
        axis x line=top, ytick pos=bottom]     
     \addplot  +[mark=none, thick] table [x="U0", y="Points:1", col sep=comma] {\figpath/re1000-2d-p1-cart1_5-strong-bc-ux-vertical-centerline.txt};
      \addplot +[mark=none, thick] table [x="U0", y="Points:1", col sep=comma] {\figpath/re1000-2d-p1-hexa3_5-strong-bc-ux-vertical-centerline.txt};
      \addplot +[mark=none, thick] table [x="U0", y="Points:1", col sep=comma] {\figpath/re1000-2d-p1-kershaw2_5-strong-bc-ux-vertical-centerline.txt};
      \addplot +[mark=o, only marks, mark size=1.4] table [x=Re1000, y=y] {\figpath/ghia-ghia-shin-ux-vertical-centerline.txt};
      \addplot +[mark=+, only marks, mark size=1.4] table [x=Re1000, y=y] {\figpath/erturk-corke-gokcol-ux-vertical-centerline.txt}; 
      \legend{%
        {$k=1$, cartesian mesh},%        
        {$k=1$, hexagonal mesh},%
        {$k=1$, Kershaw mesh},%
        Ghia et al.,%
        Erturk et al.
      }
    \end{axis}
    \begin{axis}[height=9cm, width=9cm, %
        xmin=0, xmax=1, ymin=-1, ymax=1,%
        xlabel={$x_1$}, ylabel={$u_2$}, %
        axis y line=right, xtick pos=left]
      \addplot +[mark=none, thick] table [x="Points:0", y="U1", col sep=comma] {\figpath/re1000-2d-p1-cart1_5-strong-bc-uy-horizontal-centerline.txt};
      \addplot +[mark=none, thick] table [x="Points:0", y="U1", col sep=comma] {\figpath/re1000-2d-p1-hexa3_5-strong-bc-uy-horizontal-centerline.txt};
      \addplot +[mark=none, thick] table [x="Points:0", y="U1", col sep=comma] {\figpath/re1000-2d-p1-kershaw2_5-strong-bc-uy-horizontal-centerline.txt};
      \addplot +[mark=o, only marks, mark size=1.4] table [x=x, y=Re1000] {\figpath/ghia-ghia-shin-uy-horizontal-centerline.txt};
      \addplot +[mark=+, only marks, mark size=1.4] table [x=x, y=Re1000] {\figpath/erturk-corke-gokcol-uy-horizontal-centerline.txt};      
    \end{axis}
    \begin{scope}
      \spy on (0.85,5.00) in node [fill=none, anchor=south east] at (-0.75,4.25);
      \spy on (5.50,7.00) in node [fill=none, anchor=south west] at (8.35,4.25); 
      \spy on (6.70,2.05) in node [fill=none, anchor=south west] at (8.35,0.25);
      \spy on (2.30,1.15) in node [fill=none, anchor=south east] at (-0.75,0.25);
    \end{scope}
  \end{tikzpicture}
  \caption{Two-dimensional lid-driven cavity flow, horizontal component $u_1$ of the velocity along the vertical centerline $x_1=\frac12$ and the vertical component $u_2$ of the velocity along the horizontal centerline $x_2=\frac12$ for $\Reynolds=\pgfmathprintnumber{1000}$.\label{fig:cavity:Re1000}}
\end{figure}
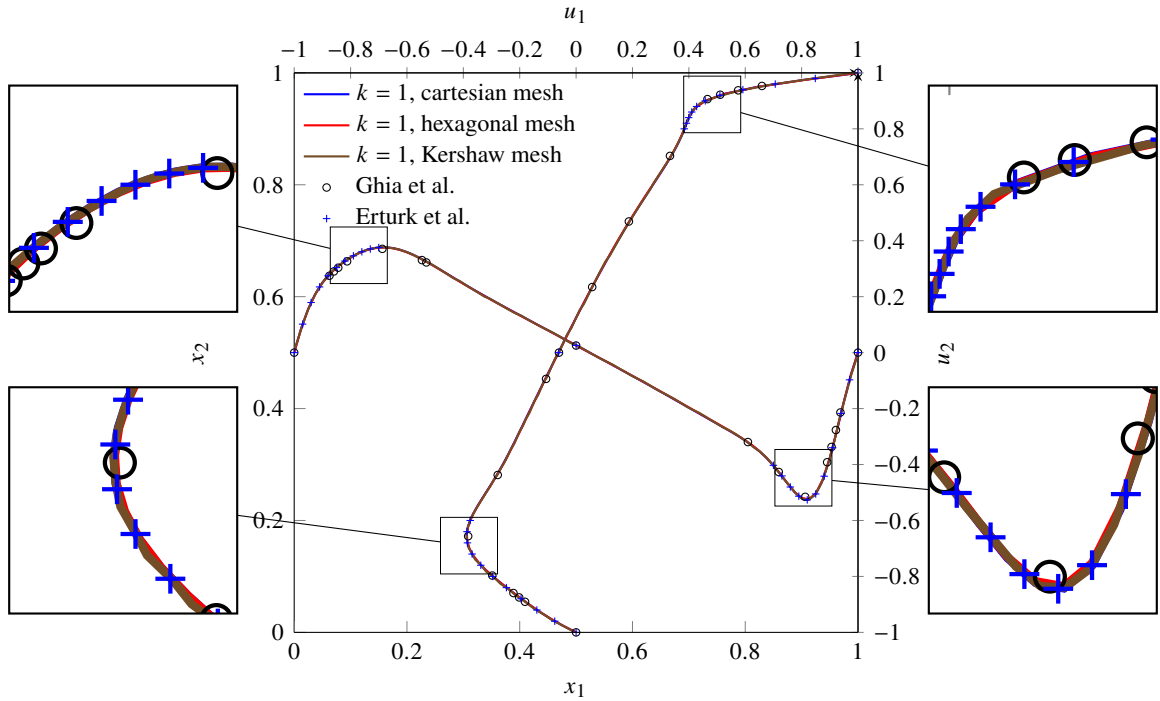

%------------------------------------------------------------------------------%
% Re = 1000 Comparison
%------------------------------------------------------------------------------%

\begin{figure}\centering
  \begin{tikzpicture}[font=\footnotesize]
    \begin{axis}[height=9cm, width=9cm, %
        xmin=-1, xmax=1, ymin=0, ymax=1, %
        xlabel={$u_1$}, ylabel={$x_2$}, %
        legend cell align={left},
        legend style = { at={(1.25,1)}, anchor=north west },
        axis x line=top, ytick pos=bottom]

      \addplot +[mark=none, thick, color=blue] table [x="U0", y="Points:1", col sep=comma] {\figpath/re1000-2d-p1-cart1_5-strong-bc-ux-vertical-centerline-lambda-10-6.txt};
      \addplot +[mark=none, thick, color=green] table [x="U0", y="Points:1", col sep=comma] {\figpath/re1000-2d-p1-hexa3_5-strong-bc-ux-vertical-centerline-lambda-10-6.txt};
      \addplot +[mark=none, thick, color=black] table [x="U0", y="Points:1", col sep=comma] {\figpath/re1000-2d-p1-kershaw2_5-strong-bc-ux-vertical-centerline-lambda-10-6.txt};
      \addplot +[mark=none, thick, color=red] table [x="U0", y="Points:1", col sep=comma] {\figpath/re1000-2d-p1-cart1_5-strong-bc-ux-vertical-centerline-npr-lambda-10-3.txt};
      \addplot +[mark=o, only marks, mark size=1.4,color=black] table [x=Re1000, y=y] {\figpath/ghia-ghia-shin-ux-vertical-centerline.txt};
      \addplot +[mark=+, only marks, mark size=1.4,color=blue] table [x=Re1000, y=y] {\figpath/erturk-corke-gokcol-ux-vertical-centerline.txt}; 
 
      \legend{%
        {$\lambda=10^6$, Cartesian, present m.},%                
        {$\lambda=10^6$, hexagonal, present m.},%                
        {$\lambda=10^6$, Kershaw, present m.},%                
        {$\lambda=10^3$, Cartesian, \cite{Botti.Di-Pietro.ea:19}},%
        Ghia et al.,%
        Erturk et al.
      };  
    \end{axis}
    \begin{axis}[height=9cm, width=9cm, %
        xmin=0, xmax=1, ymin=-1, ymax=1,%
        xlabel={$x_1$}, ylabel={$u_2$}, %
        axis y line=right, xtick pos=left]

      \addplot +[mark=none, thick, color=blue] table [x="Points:0", y="U1", col sep=comma] {\figpath/re1000-2d-p1-cart1_5-strong-bc-uy-horizontal-centerline-lambda-10-6.txt};      
      \addplot +[mark=none, thick, color=green] table [x="Points:0", y="U1", col sep=comma] {\figpath/re1000-2d-p1-hexa3_5-strong-bc-uy-horizontal-centerline-lambda-10-6.txt};      
      \addplot +[mark=none, thick, color=black] table [x="Points:0", y="U1", col sep=comma] {\figpath/re1000-2d-p1-kershaw2_5-strong-bc-uy-horizontal-centerline-lambda-10-6.txt};      
      \addplot +[mark=none, thick, color=red] table [x="Points:0", y="U1", col sep=comma] {\figpath/re1000-2d-p1-cart1_5-strong-bc-uy-horizontal-centerline-npr-lambda-10-3.txt};      
      \addplot +[mark=o, only marks, mark size=1.4,color=black] table [x=x, y=Re1000] {\figpath/ghia-ghia-shin-uy-horizontal-centerline.txt};
      \addplot +[mark=+, only marks, mark size=1.4,color=blue] table [x=x, y=Re1000] {\figpath/erturk-corke-gokcol-uy-horizontal-centerline.txt};      
 
    \end{axis}
  \end{tikzpicture}
  \caption{Two-dimensional lid-driven cavity flow with irrotational force $\bef = \lambda \GRAD \psi$ with $\lambda =10^6$.
    Comparison between the present method and the original HHO formulation of \cite{Botti.Di-Pietro.ea:19} both using $k=1$.
    The plot represents the horizontal component $u_1$ of the velocity along the vertical centerline $x_1=\frac12$ and the vertical component $u_2$ of the velocity along the horizontal centerline $x_2=\frac12$ for $\Reynolds=\pgfmathprintnumber{1000}$.\label{fig:cavity:Re1000:wlambda}}
\end{figure}
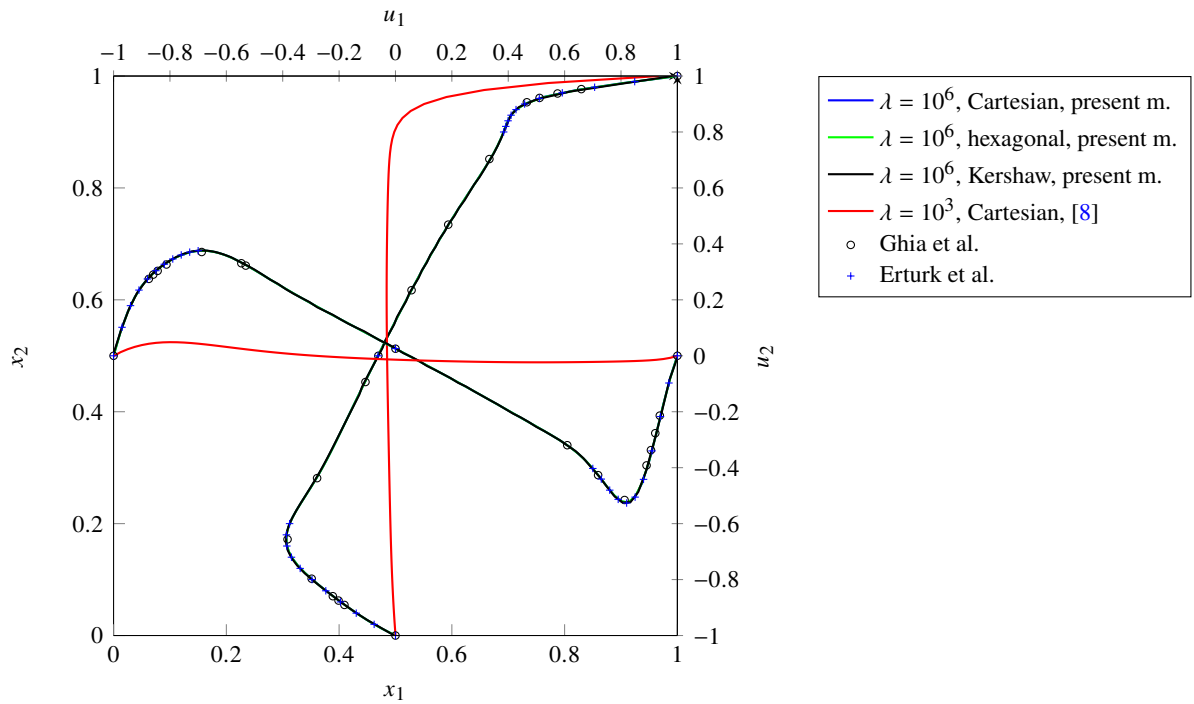

%------------------------------------------------------------------------------%

\section*{Acknowledgements}
{The authors are grateful for the suggestions of an anonymous referee which contributed to improving the quality of the manuscript.}
Daniele Di Pietro acknowledges the partial support of \emph{Agence Nationale de la Recherche} grant ANR-20-MRS2-0004 ``NEMESIS'' and I-Site MUSE grant ANR-16-IDEX-0006 ``RHAMNUS''.

%------------------------------------------------------------------------------%
\printbibliography 
%------------------------------------------------------------------------------%

\appendix

\section{Proof of Lemma \ref{lemma:appx}}\label{sec:appx}

\begin{proof}
{For this proof we take inspiration mainly from \cite[Section 3]{Kreuzer:2021}.}
  First of all, let us introduce a few new definitions. 
  We denote by $\hat{\tau}$ the reference tetrahedron. From the assumptions of Section \ref{sec:setting:mesh}, there exists $\bx_T\in \mathbb{R}^3$ which is a common vertex
  for all simplices in $\Ths[T]$.
  Then, for each $\tau\in\Ths[T]$, it is possible to construct a one-to-one affine map $\bF_\tau: \hat{\tau} \rightarrow \tau$ such that
  \begin{equation}\label{eq:appx.affinemap}
    \bF_\tau= \mxJ_\tau\hat{\bx} + \bx_T,
  \end{equation}
  where $\mxJ_\tau$ is an invertible real matrix of size $3\times3$.
  Now,
  given $\tau\in \Ths[T]$ and {$\hat{\bv},\hat{\bw}\in\Poly{l}(\hat{\tau})^3$} where {$l\geq 0$} we introduce, respectively, the contravariant and covariant Piola's transformations (see \cite{Ern.Guermond:21:I}) as follows,
  \begin{align}\label{def:piola}
    \bpsi_{d,\tau}(\hat{\bv})\coloneq |\det \mxJ_\tau|^{-1} \mxJ_\tau (\hat{\bv} \circ \bF_\tau^{-1})
    &&\text{and}&&
    \bpsi_{c,\tau}(\hat{\bv})\coloneq \mxJ_\tau^{-\text{T}} (\hat{\bv} \circ \bF_\tau^{-1}),
  \end{align}
  which crucially satisfy
  \begin{equation}\label{eq:int.Piola}
    \int_\tau 
    \bpsi_{d,\tau}(\hat{\bv})\cdot\bpsi_{c,\tau}(\hat{\bw})
    =
    \int_{\hat{\tau}} \hat{\bv}\cdot\hat{\bw}.
  \end{equation}
  %%In addition, it is well known that $\bpsi_{d,\tau}(\cdot)$ satisfies
  %%\begin{equation}\label{eq:div.dPiola}
  %%\DIV\bpsi_{d,\tau}(\hat{\bv}) =|\det \mxJ_\tau|^{-1}\DIV(\hat{\bv}\circ\bF_\tau^{-1}).
  %%\end{equation}

  We now introduce the space $\Goly{{\rm c},k-1}(\hat{\tau})\coloneq\hat{\bx}\times\Poly{k-2}(\hat{\tau})^3$, and the operator
  $\hat{\bE}^{k-1}_{\hat{\tau}}: L^2(\hat{\tau})^3 \rightarrow \Goly{{\rm c},k-1}(\hat{\tau})$
  such that, for all $\hat{\bv}\in L^2(\hat{\tau})^3$ and all $\hat{\bg}\in\Goly{{\rm c},k-1}(\hat{\tau})$,
  \begin{equation}\label{eq:appx.def.Ehat}
    \int _{\hat{\tau}} \hat{\lambda}^2{\ROTh}\hat{\bE}^{k-1}_{\hat{\tau}}(\hat{\bv})\cdot{\ROTh}\hat{\bg} =\int _{\hat{\tau}}\hat{\bv}\cdot\hat{\bg}, 
  \end{equation}
  where $\hat{\lambda}$ is defined as the product of all the barycentric coordinates of ${\hat{\bx}}$ in $\hat{\tau}$, i.e., $\hat{\lambda}=\prod_{i=1}^4 \hat{\lambda}_i$.
  The fact that \eqref{eq:appx.def.Ehat} defines $\hat{\bE}^{k-1}_{\hat{\tau}}(\hat{\bv})$ uniquely follows from the Riesz representation theorem after observing that $\ROTh:\Goly{{\rm c},k-1}(\hat{\tau})\to{\Poly{k-2}(\hat{\tau})^3}$ is an isomorphism.
  We then define the operator
  ${\bE}^{k-1}_{\tau}: L^2(\tau)^3 \rightarrow \Poly{k+5}(\tau)^3\cap H^1_0(\tau)^3$ as follows
  (see \cite[Eq. (45)]{Kreuzer:2021})
  \begin{equation}\label{eq:appx.def.E}
    {\bE}^{k-1}_{\tau}(\bw)\coloneq \bpsi_{d,\tau}\left(
    {\ROTh}\left[%
      \hat{\lambda}^2\ROTh\hat{\bE}^{k-1}_{\hat{\tau}}(\bpsi^{-1}_{d,\tau}(\bw))
      \right]
    \right).
  \end{equation}
  %where $\bpsi^{-1}_{d,\tau}(\cdot)$ is the inverse of the contravariant transformation which is expressed as $\bpsi^{-1}_{d,\tau}(\bv)=|\det \mxJ_\tau| \mxJ_\tau^{-1} ({\bv} \circ \bF_\tau)$. 
  Using standard properties of the contravariant transformation $\bpsi_{d,\tau}(\cdot)$, we infer $\DIV{\bE}^{k-1}_{\tau}(\bw)=0$; moreover,
  it is proved in \cite[Proof of Proposition 17]{Kreuzer:2021} that $\norm{{\Ldeuxd[\tau]}}{\bE^{k-1}_{\tau}(\bw)}\lesssim\norm{{\Ldeuxd[\tau]}}{\bw}$. 
  Then, for given a function $\bv\in L^2(T)^3$,
  we define $\bvs_{0,\tau}$ as the interpolate of $\bE^{k-1}_{\tau}(\bv)$ onto the space
  $\RTN{k}(\tau)$, thus we have $\DIV\bvs_{0,\tau}=0$. Additionally, since $\bE^{k-1}_{\tau}(\bv)\in H^1_0(\tau)^3$, $\bvs_{0,\tau}$
  has zero normal trace at the boundary of $\tau$ and,
  using standard interpolation estimates for $\bvs_{0,\tau}$ (see, e.g., \cite[Proposition 2.5.1]{Boffi.Brezzi.ea:13}),
  the bound $\norm{{\Ldeuxd[\tau]}}{\bE^{k-1}_{\tau}(\bv)}\lesssim\norm{{\Ldeuxd[\tau]}}{\bv}$,
  and  a discrete  inverse inequality (this is valid since $\bE^{k-1}_{\tau}(\bv)$ is a polynomial function), it is inferred that
  $\norm{{\Ldeuxd[\tau]}}{\bvs_{0,\tau}}\lesssim\norm{{\Ldeuxd[\tau]}}{\bv}$.
  We now define $\widetilde{\bR}_T^k(\bv)\in\RTN{k}_0(\Ths[T])$ as 
  \begin{equation}\label{def:ws0}
    \widetilde{\bR}_T^k(\bv)_{|\tau} \coloneq \bvs_{0,\tau} \qquad \forall \tau \in \Ths[T],
  \end{equation}
  and observe that $\widetilde{\bR}_T^k(\bv)$ clearly satisfies the properties (\ref{eq:rtn:golyc:b}--\ref{eq:rtn:golyc:d}) from the above discussion.
  \smallskip

  To prove \eqref{eq:rtn:golyc:a},  we  introduce the space
  \begin{equation}\label{def:goly.sT}
    \Goly{{\rm c},k-1}(\Ths[T])\coloneq (\bx-\bx_T) \times \Poly{k-2}(\Ths[T])^3,
  \end{equation}
  and denote the $L^2$-orthogonal projector onto this space by $\bpi^{{{\rm c},k-1}}_{\Goly{},\Ths[T]}$.
  Observe that $\Goly{{\rm c},k-1}(T)\subset\Goly{{\rm c},k-1}(\Ths[T])$, 
  thus 
  $$
  \bpi^{{\rm c},k-1}_{\Goly{},T} ( \bpi^{{\rm c},k-1}_{\Goly{},\Ths[T]}\bw)
  = \bpi^{{\rm c},k-1}_{\Goly{},T}\bw
  \qquad\forall \bw\in L^2(T),
  $$
  then \eqref{eq:rtn:golyc:a} holds a fortiori if we prove that
  \begin{equation}\label{eq:appx:toprove}
    \bpi^{{\rm c},k-1}_{\Goly{},\Ths[T]}\widetilde{\bR}_T^k(\bv)
    = \bpi^{{\rm c},k-1}_{\Goly{},\Ths[T]}\bv.
  \end{equation}
  To prove it, let $\bg\in \Goly{{\rm c},k-1}(\Ths[T])$ and
  $\tau\in\Ths[T]$.
  Then, using the definition \eqref{def:ws0} of $\widetilde{\bR}_T^k(\bv)$, we obtain, for all $\tau\in\Ths[T]$,
  \begin{equation}\label{eq:apx:I}
    \int _\tau \widetilde{\bR}_T^k(\bv)\cdot \bg 
    =\int _\tau  \bvs_{0,\tau} \cdot \bg 
    =\int _\tau  \bE^{k-1}_{\tau}(\bv)\cdot \bg 
    =
    \int_{\hat{\tau}}
    \bpsi^{-1}_{d,\tau} (\bE^{k-1}_{\tau}(\bv))\cdot \bpsi^{-1}_{c,\tau}(\bg)
    \eqcolon\mathfrak{T},
  \end{equation}
  where in the second step we have used the interpolation properties of $\bvs_{0,\tau}$ and the fact that $\bg_{|\tau}\in\Poly{k-1}(\tau)^3$ along with the definition of the Raviart--Thomas interpolator, and in the last step the definitions \eqref{def:piola} of the Piola transformations and the identity \eqref{eq:int.Piola}.
  By definition \eqref{def:goly.sT},
  we have that $\bg_{|\tau}=(\bx-\bx_T)\times \bq_{\bg}$ where $\bq_{\bg}\in\Poly{k-2}(\tau)^3$.
  With this mind, and using \eqref{eq:appx.affinemap}, we get that 
  \begin{equation}\label{eq:surj.cpiola}
    \bpsi^{-1}_{c,\tau}(\bg)(\hat{\bx})
    =(\mxJ_\tau^{-1})^{-\text{T}}\left( \mxJ_\tau\hat{\bx} \times \bq_g(\bF_\tau(\hat{\bx}))\right)
    = \hat{\bx} \times \det(\mxJ_\tau)\mxJ_\tau^{-1} \bq_g(\bF_\tau(\hat{\bx})),
  \end{equation}
  where in the last step we have used the matrix-cross-product identity
  (see \cite[Ex.9.5]{Ern.Guermond:21:I}) $\mxA^{-\text{T}}(\by\times\bz)=\det (\mxA)^{-1}\mxA\by\times\mxA\bz$ valid for any $\by,\bz\in\mathbb{R}^3$ and any invertible real matrix $\mxA$ of size $3\times3$.
  Now, define
  $\widetilde{\bq}_g\coloneq\det(\mxJ_\tau)\mxJ_\tau^{-1} \bq_g(\bF_\tau(\hat{\bx}))\in\Poly{k-2}(\hat{\tau})^3$.
  Then, using the definition \eqref{eq:appx.def.E}, we compute $\mathfrak{T}$ in \eqref{eq:apx:I} as follows:
  \begin{align*}
    \mathfrak{T}
    &=
    \int_{\hat{\tau}}
    {\ROTh}\left(\hat{\lambda}^2{\ROTh}\hat{\bE}^{k-1}_{\hat{\tau}}(\bpsi^{-1}_{d,\tau}(\bv))\right)
    \cdot
    (\hat{\bx}\times \widetilde{\bq}_g)\\
    &=
    \int_{\hat{\tau}}
    \hat{\lambda}^2{\ROTh}\hat{\bE}^{k-1}_{\hat{\tau}}(\bpsi^{-1}_{d,\tau}(\bv))
    \cdot
    {\ROTh}{} (\hat{\bx}\times \widetilde{\bq}_g)\\
    &=
    \int_{\hat{\tau}}
    \bpsi^{-1}_{d,\tau}(\bv)
    \cdot
    (\hat{\bx}\times \widetilde{\bq}_g)
    =
    \int_{\tau}
    \mxJ^{-1}_\tau\bv
    \cdot
    \mxJ^{\text{T}}_\tau
    ((\bx-\bx_T)\times{\bq}_g)
    =
    \int_{\tau}
    \bv
    \cdot
    \bg,
  \end{align*}
  where, in the second line, we have used integration by parts,
  along with the fact that $\hat{\lambda}^2{\ROTh}\hat{\bE}^{k-1}(\cdot)$ vanishes at the boundary of $\hat{\tau}$,
  in the third line first the definition \eqref{eq:appx.def.Ehat},
  then a change of coordinates using \eqref{eq:appx.affinemap} along with the definitions \eqref{def:piola}
  and the same matrix-cross-product identity as before, and finally some standard  properties of the transpose.
  Thus, using the last equation above and \eqref{eq:apx:I}, we have that 
  \begin{equation}\label{eq:appx.lasteq}
    \int_T\widetilde{\bR}_T^k(\bv)\cdot\bg
    =\sum_{\tau\in\Ths[T]}\int_\tau\widetilde{\bR}_T^k(\bv)\cdot\bg
    =\sum_{\tau\in\Ths[T]}\int_\tau\bv\cdot\bg
    =\int_T\bv\cdot\bg.
  \end{equation}
  Since $\bg$ is an arbitrary element of $\Goly{{\rm c},k-1}(\Ths[T])$, it implies \eqref{eq:appx:toprove}, and we conclude.
\end{proof}

\begin{remark}[The common vertex  assumption]\label{rem:cvertex.prop}
  In  the proof of Lemma \ref{lemma:appx}, the fact that $ \bx_T$ is a common vertex for all simpex $\tau \in \Ths[T]$ allows to express the affine transformation $\bF_\tau: \hat{\tau} \rightarrow \tau$  as \eqref{eq:appx.affinemap} implying the key property  that the covariant transformation $\bpsi_{c,\tau}:\Goly{{\rm c},k-1}(\hat{\tau}) \rightarrow ({\bx}-\bx_T)\times\Poly{k-2}({\tau})^3$,  defined in  \eqref{def:piola}, is an isomorphism. This  is required for \eqref{eq:surj.cpiola} and \eqref{eq:appx.lasteq}, and then  making possible to prove \eqref{eq:rtn:golyc:a} and \eqref{eq:rtn:golyc:d}. 
\end{remark}

%------------------------------------------------------------------------------%

\end{document}